\documentclass{amsart}%
\usepackage{amssymb}
\usepackage{amsfonts}%
\usepackage{amsmath}%
\usepackage{graphicx}
\providecommand{\U}[1]{\protect \rule{.1in}{.1in}}
\newtheorem{theorem}{Theorem}
\theoremstyle{plain}

\newtheorem{proposition}{Proposition}
\newtheorem{remark}{Remark}

\numberwithin{equation}{section}
\usepackage[utf8]{inputenc}
\title{ Approximation Results on Neural Network  Operators of Convolution Type }
\begin{document}
\author{Asiye Arif$^{1}$}
\author{Tu\u{g}ba Yurdakadim$^{2}$}
\subjclass[]{}
\maketitle
\noindent$^{1,2}$Bilecik Şeyh Edebali University, Turkey\\
E-mail: arif\_asiya@mail.ru$^1$, tugba.yurdakadim@bilecik.edu.tr $^2$\\[2ex]

\begin{abstract}
In the present paper, we introduce three neural network operators of  convolution type  activated by symmetrized, deformed and parametrized $ B $- generalized logistic  function. We deal with  the approximation  properties of these  operators to the identity by using  modulus of continuity. Furthermore, we show that  our operators preserve global smoothness and  consider the iterated versions of them. Here, we find it is worthy to mention that these operators play important roles in neural network approximation since most of the basic  network models are activated  by logistic functions. 
\end{abstract}

\section{Introduction and Motivation}
 \indent The inspiration behind neural networks, artificial neural networks is biological nature and function of the human brain. Neural and artificial neural networks  are cornerstone of modern machine learning, artificial  intelligence. The investigation to create artificial intelligence capable of learning and problem solving has a rich history with early conceptualizations dating back to the perception in the late 1950's \cite{rozenbaltt}. This pioneering work has laid the groundwork for future developments by introducing the concept of a simple learning algorithm. However, this field has faced significant challenges and periods of reduced interest, often referred as ''AI winters", due to limitations in computational power and the complexity of the problems being considered. Despite these setbacks, the fundamental ideas of interconnected nodes (neurons) processing information and adjusting their connections (weight) based on experience stay down in the center of field's desires. \\
 \indent The resurgence of neural networks in recent decades is largely attributed to advancements in the power of computers, accessibility of big data, and breakthroughs in network architectures and training algorithms. This renewed interest has led to rise of machine learning, a sub-field of AI focuses  to enable computers  learn from data without explicit programming and  encompasses  a broad  range of techniques consisting of neural networks, decision trees, support vector machines, Bayesian methods. Within this landscape in data, neural networks have proven particularly adept at handling complex, non-linear relationships in data, making them appropriate for particular tasks, for example, image recognition, natural language processing, optimization, process control system, forecasting, approximation of functions and solutions of curve fitting \cite{ALADAg},  \cite{ARMSTRONG}, \cite{AVCI}, \cite{amatod}, \cite{BISHOP},  \cite{CHUNG}, \cite{CURTIS}, \cite{cIFTER}, \cite{FAALSIDE},  \cite{MISNKY}, \cite{PARKER}.\\
 \indent A significant leap forward within machine learning has been the emergence of deep learning. Deep learning models are characterized by their depth, meaning they consist of multiple layers of interconnected neurons \cite{FAUSETT}. This depth allows them to learn hierarchical representations of data, with each layer extracting increasingly abstract and complex features. For instance, in image recognition, the first layers might detect edges and corners, while deeper layers learn to recognize shapes, objects, and ultimately, entire scenes. The success of deep learning has been fueled by the development of efficient training algorithms, such as backpropagation, and the availability of specialized hardware, like GPUs, which can accelerate the training process  \cite{rozenbaltt}, \cite{FAUSETT}, \cite{HINTON} \cite{Seppo}. Convolutional neural networks, recurrent neural networks, and transformers are just a few examples of deep learning architectures that have reconfigured fields like computer vision, natural language processing, and speech recognition. Furthermore, the exploration of novel activation functions and approximation operators plays  crucial roles in enhancing the performance and efficiency of these models. \\
 \indent Motivated by the capabilities of neural networks, recent researches have focused on the rate  of neural network approximation to continuous functions by special  neural network operators of Cardaliaguet-Euvrard and squashing types, with the use of  modulus of continuity of  given function or its high order derivative both in univariate and multivariate cases. Furthermore, this direction of studies in this area have been developed by considering different activation functions 
\cite{G.A.2023}, \cite{G.A.2022}, \cite{Anastassiou20232}, \cite{anast 1997}, \cite{anast 2001}, \cite{anast 2011}, \cite{banahc}, \cite{cantur}, \cite{BROOMHEAD}, \cite{Cardaliaguet}, \cite{CASDAGLI}, \cite{Cybenko},  \cite{POWELL}.
\\
 \indent In the present paper, we introduce three neural network operators of  convolution type  activated by symmetrized, deformed and parametrized $ B $- generalized logistic  function. We deal with  the approximation  properties of these  operators to the identity by using  modulus of continuity. Furthermore, we show that  our operators preserve global smoothness and  consider the iterated versions of them. Here, we find it is worthy to mention that these operators play important roles in neural network approximation since most of the basic  network models are activated  by logistic functions.
\section{ Preliminaries}
Let us recall the following deformed, parametrized, $B-$ generalized logistic function and its basic properties as below
\cite{G.A.2023}:
\begin{equation}
\nu_{q,\beta }(x)=\frac{1}{1+qB^{-\beta  x}},~ x \in \mathbb{R},~ q,~\beta >0,~ B>1
\end{equation}
\begin{itemize}
\item[•] $\nu_{q,\beta }(+\infty)=1, ~~\nu_{q,\beta }(-\infty)=0,$
\item[•] $\nu_{q,\beta }(x)=1-\nu_{\frac{1}{q},\beta }(-x),~ for~every~ x \in (-\infty, \infty),$
\item[•] $ \nu_{q,\beta }(0)=\frac{1}{1+q}.$
\end{itemize}
Then with the use of $\nu$, the following activation function  has been considered  and its properties have been given in \cite{G.A.2022}:
\begin{equation}
G_{q,\beta }(x)=\frac{1}{2}\left(\nu_{q, \beta }(x+1)-\nu_{q,\beta }(x-1)\right),~x \in (-\infty, \infty).
\end{equation}
One can easily notice that 
\begin{align*}
G_{q, \beta }(-x)=&\frac{1}{2}\left(\nu_{\frac{1}{q},\beta }(x+1)-\nu_{\frac{1}{q}, \beta }(x-1)\right) =G_{ \frac{1}{q}, \beta }(x), 
\end{align*}

i.e.,
\begin{equation}
\label{symmetry}
G_{q,\beta }(-x)=G_{ \frac{1}{q}, \beta }(x), 
\end{equation} 
which means, that we have a deformed symmetry and the global maximum of $G_{q,\beta}(x)$ is
\begin{equation*}
G_{q,  \beta }\left( \frac{\log_B q}{\beta} \right)= \frac{B^{\beta }-1}{2(B^{ \beta }+1)}.
\end{equation*}
\noindent It is known from \cite{Anastassiou20232} that  $ G_{q, \beta }$ is a density function on $(-\infty, \infty)$ since,
\begin{equation*}
\int_{-\infty}^{\infty} G_{q, \beta }(x)dx=1.
\end{equation*}

By observing the following symmetry 
\begin{equation*}
\left( G_{q,\beta } + G_{ \frac{1}{q}, \beta } \right)(-x)=\left( G_{q, \beta } + G_{ \frac{1}{q}, \beta } \right)(x),
\end{equation*}
we introduce 
\begin{equation*}
\Psi= \frac{ G_{q, \beta }+ G_{ \frac{1}{q}, \beta }}{2}.
\end{equation*}
Since our aim is to get a symmetrized density function, this definition of $\Psi$ serves to our aim, i.e.,
 \begin{equation*}
 \Psi(x)=\Psi(-x),~ x \in (-\infty, \infty)
 \end{equation*}
 and
\begin{equation*}
\int_{-\infty}^{ \infty} \Psi(x)dx=1, 
\end{equation*}
implying
\begin{equation}
\int_{-\infty}^{\infty} \Psi(nx-v)dv=1, for~every~ n \in \mathbb{N},~ x \in (-\infty, \infty).
\end{equation}

\section{Main Results}
Here, we introduce the following convolution type operators activated by symmetrized, $ q $-deformed and $ \beta $-parametrized $ B $- generalized logistic function. Our aim in this section is to examine their approximation properties and improve these results. \\
\indent Now by considering the above new density function $\Psi $, we introduce the following neural network operators of convolution type for $f\in C_{B}(\mathbb{R})$ consisting of the functions which are continuous and bounded on $\mathbb{R}$ and $n\in \mathbb{N}$ :
\begin{equation}
\label{3.1}
\mathtt{A}_n (f)(x):=\int_{-\infty}^{ \infty} f \left( \frac{v}{n}\right) \Psi(nx-v) dv,
\end{equation}
\begin{align}
\label{3.2}
\mathtt{A}_n^*(f)(x):&= n \int_{-\infty}^{ \infty} \left( \int^{\frac{v+1}{n}}_{\frac{v}{n}} f \left( h \right)dh\right)\Psi(nx-v) dv\\
\nonumber &=n \int_{-\infty}^{\infty} \left( \int^{\frac{1}{n}}_0 f \left( h+\frac{v}{n}\right)dh\right)\Psi(nx-v) dv,
\end{align}
\begin{equation}
\label{3.3}
 \overline{\mathtt{A}_n}(f)(x):=\int_{-\infty}^{ \infty} \left( \sum_{s=1}^r w_s f \left( \frac{v}{n}+\frac{s}{nr}\right)\right) \Psi(nx-v) dv,
\end{equation}
where $w_s \geq 0$, $\displaystyle \sum_{s=1}^r w_s=1$ and 
(\ref{3.2}) is  called  activated Kantorovich type of (\ref{3.1}),
(\ref{3.3}) is  called  activated  Quadrature type of (\ref{3.1}).\\
Here is our first result. 
\begin{theorem}
\label{t1}
\begin{equation}
\int_{\{v\in\mathbb{R}:|nx-v|\geq n^{1-\alpha}\}}\Psi(nx-v) dv< \frac{\left( q+\frac{1}{q} \right) }{B^{ \beta (n^{1-\alpha}-1)}}
\end{equation} holds for  $0<\alpha<1$, $n\in \mathbb{N}$ such that $n^{1-\alpha}>2$. 
\end{theorem}
\begin{proof}
Since 
\begin{equation}
G_{q, \beta }(x)=\frac{1}{2}\left(\nu_{q, \beta }(x+1)-\nu_{q, \beta }(x-1)\right),~x \in (-\infty, \infty),
\end{equation}
first letting 
$ x \in [1, \infty)$, and using Mean Value Theorem, we get 
\begin{equation}
G_{q, \beta }(x)=\frac{1}{2} \nu'_{q, \beta }(\xi)2=\nu'_{q, \beta }(\xi)=\frac{q  \beta (\ln B)}{(1+q  B^{-\beta  \xi})^2  B^{\beta \xi}},~ 0\leq x-1<\xi< x+1.
\end{equation}
Also, we obtain  that 
\begin{equation*}
G_{q, \beta }(x)=q  \beta  (\ln B) B^{-\beta  \xi}< q  \beta  (\ln B) B^{-\beta (x-1)}
\end{equation*}
and similarly,  
\begin{equation*}
G_{ \frac{1}{q}, \beta }(x)=\frac{1}{q} \beta  (\ln B) B^{-\beta  \xi}< \frac{1}{q}  \beta  (\ln B) B^{-\beta (x-1)}
\end{equation*}
which imply 
\begin{equation*}
\Psi<\frac{1}{2}\left( q+\frac{1}{q} \right)  \beta  (\ln B) B^{-\beta (x-1)}.
\end{equation*}
Let 
\begin{equation*}
F:=\{ v\in \mathbb{R}:|nx-v|\geq n^{1-\alpha} \}
\end{equation*}
and then we can write that 
\begin{align*}
\int_F \Psi(nx-v) dv&=\int_F \Psi(|nx-v|) dv<\frac{1}{2}\left( q+\frac{1}{q} \right) \beta (\ln B) \int_F B^{-\beta (|nx-v|-1)} dv\\
&=2 \cdot \frac{1}{2}\left( q+\frac{1}{q} \right) \beta  (\ln B)\int_{n^{1-\alpha}}^{\infty} B^{-\beta (x-1)}dx\\
&=\frac{\left( q+\frac{1}{q} \right)}{B^{\beta (n^{1-\alpha}-1)}}.
\end{align*}
Hence, we complete the proof. 
\end{proof}
Notice that  if one takes $q=1$, $\beta =\mu$, $B=e$ then Theorem \ref{t1} reduces to Theorem 5 of \cite{G.A.2024}.
Recall the modulus of continuity  of a function $f$ for $\theta >0$ on $\mathbb{R}$ as follows:
\begin{displaymath}
\omega(f,\theta):=\sup_{ \begin{subarray}{l}
x,y \in \mathbb{R}\\
|x-y|\leq\theta \end{subarray} } |f(x)-f(y)|.
\end{displaymath}
Now we are ready to present quantitative convergence results for our operators. Note that for Theorems 2, 3, 4, we let  $0<\alpha<1$, $n\in \mathbb{N}$ such that $n^{1-\alpha}>2$. 
\begin{theorem}
\label{theorem3}
\begin{equation*}
\left| \mathtt{A}_n (f)(x)-f(x)\right|\leq\omega\left(f,\frac{1}{n^{\alpha}}\right)+ \frac{2\left( q+\frac{1}{q} \right)  \Vert f \Vert_{\infty}}{B^{ \beta (n^{1-\alpha}-1)}}=:\mathtt{T}
\end{equation*}
and 
\begin{equation*}
\Vert \mathtt{A}_n(f)-f\Vert_{\infty} \leq \mathtt{T}
\end{equation*} hold for   $f\in C_B(\mathbb{R})$. 
Furthermore,  we have  that $ \displaystyle \lim_{n\rightarrow \infty} \mathtt{A}_n(f)=f$, pointwise and uniformly for $f\in C_{UB}(\mathbb{R})$, consisting of the functions which are uniformly continuous and bounded on $\mathbb{R}$.
\end{theorem} 
\begin{proof}
Let
$ \displaystyle F_1:=\left\lbrace v\in \mathbb{R}:\left| \frac{v}{n}-x \right|<\frac{1}{n^{\alpha}} \right \rbrace$,
$ \displaystyle F_2:=\left\lbrace v\in \mathbb{R}:\left| \frac{v}{n}-x \right| \geq \frac{1}{n^{\alpha}} \right \rbrace.$
Then $F_1\cup F_2=\mathbb{R}$ and we can write by using Theorem \ref{t1} that 
\begin{align*}
\vert \mathtt{A}_n(f)(x)-f(x)\vert&=\left\vert \int_{-\infty}^{ \infty} f\left(\frac{v}{n}\right) \Psi(nx-v)dv-f(x)\int_{-\infty}^{ \infty} \Psi(nx-v)dv \right\vert\\
&\leq  \int_{-\infty}^{ \infty} \left\vert f\left(\frac{v}{n}\right)-f(x)\right\vert \Psi(nx-v)dv\\
&= \int_{F_1} \left\vert f \left(\frac{v}{n}\right)-f(x)\right\vert \Psi(nx-v)dv+  \int_{F_2} \left\vert f \left(\frac{v}{n}\right)-f(x)\right\vert \Psi(nx-v)dv\\
&\leq \int_{F_1} \omega \left( f, \left\vert \frac{v}{n}-x \right\vert \right) \Psi(nx-v)dv+  2 \Vert f \Vert_{\infty}\int_{F_2}  \Psi(nx-v)dv\\
& \leq \omega \left( f, \frac{1}{n^{\alpha}} \right)+  \frac{  2 \Vert f \Vert_{\infty} \left( q+\frac{1}{q} \right)}{B^{\beta (n^{1-\alpha}-1)}}
\end{align*}
which gives the desired results. 
\end{proof}
\begin{theorem}
\label{theorem4}
\begin{equation*}
\left\vert \mathtt{A}_n^* (f)(x)-f(x)\right\vert \leq\omega\left(f,\frac{1}{n}+\frac{1}{n^{\alpha}}\right)+ \frac{2\left( q+\frac{1}{q} \right)\Vert f \Vert_{\infty}}{B^{\beta (n^{1-\alpha}-1)}}=:\mathtt{E}
\end{equation*} 
and 
\begin{equation*}
\Vert \mathtt{A}_n^*(f)-f\Vert_{\infty} \leq \mathtt{E}
\end{equation*} hold for $f\in C_B(\mathbb{R})$.
Furthermore, we have that $\displaystyle \lim_{n\rightarrow \infty} \mathtt{A}_n^*(f)=f$, pointwise and uniformly for  $f\in C_{UB}(\mathbb{R})$.
\end{theorem}
\begin{proof}
Let $F_1$ and $F_2$ be defined as above. Then we obtain by using Theorem \ref{t1} that
\begin{align*}
\left\vert \mathtt{A}_n^* (f)(x)-f(x)\right\vert&=\left\vert n \int_{-\infty}^{ \infty} \left( \int^{\frac{v+1}{n}}_{\frac{v}{n}} f \left( t\right)dt\right) \Psi(nx-v) dv -\int_{-\infty}^{ \infty} f(x) \Psi(nx-v)dv  \right \vert \\
&=\left\vert n \int_{-\infty}^{ \infty} \left( \int^{\frac{v+1}{n}}_{\frac{v}{n}} f \left( t\right)dt\right) \Psi(nx-v) dv -n \int_{-\infty}^{ \infty} \left(   \int^{\frac{v+1}{n}}_{\frac{v}{n}} f(x)dt \right) \Psi(nx-v)dv \right\vert \\ 
&\leq  n \int_{-\infty}^{ \infty} \left( \int^{\frac{v+1}{n}}_{\frac{v}{n}} \left\vert f(t)-f(x)\right\vert dt
\right) \Psi(nx-v) dv\\
&=  n \int_{-\infty}^{ \infty} \left( \int^{\frac{1}{n}}_{0} \left\vert f\left(t+\frac{v}{n}\right) -f(x)\right\vert dt
\right) \Psi(nx-v) dv \\
&\leq   \int_{F_1} \left( n  \int^{\frac{1}{n}}_{0} \left\vert f\left(t+\frac{v}{n}\right) -f(x)\right\vert dt
\right)  \Psi(nx-v) dv \\
&+\int_{F_2} \left( n  \int^{\frac{1}{n}}_{0} \left\vert f\left(t+\frac{v}{n}\right) -f(x)\right\vert dt
\right) \Psi(nx-v) dv\\
&\leq   \int_{F_1} \left( n  \int^{\frac{1}{n}}_{0}  \omega \left(f,|t|+\frac{1}{n^{\alpha}}\right)dt
\right) \Psi(nx-v) dv +2 \Vert f \Vert_{\infty}\int_{F_2}  \Psi(nx-v)dv\\
&\leq \omega \left(f, \frac{1}{n}+\frac{1}{n^{\alpha}}\right)+ \frac{2\left( q+\frac{1}{q} \right)  \Vert f \Vert_{\infty}}{B^{\beta (n^{1-\alpha}-1)}}
\end{align*}
which gives the desired results. 
\end{proof}
\begin{theorem}
\label{theorem5}
\begin{equation*}
\left\vert \overline{\mathtt{A}_n} (f)(x)-f(x)\right\vert \leq\omega\left(f,\frac{1}{n}+\frac{1}{n^{\alpha}}\right)+ \frac{ 2\left( q+\frac{1}{q} \right) \Vert f \Vert_{\infty}}{B^{ \beta (n^{1-\alpha}-1)}}=\mathtt{E}
\end{equation*} 
and 
\begin{equation*}
\Vert \overline{\mathtt{A}_n}(f)-f\Vert_{\infty} \leq \mathtt{E}
\end{equation*} hold for $f\in C_B(\mathbb{R})$.
Furthermore  we have  $\displaystyle \lim_{n\rightarrow \infty} \overline{\mathtt{A}_n}(f)=f$, pointwise and uniformly for  $f\in C_{UB}(\mathbb{R})$.
\end{theorem}
\begin{proof}
Again, we can write by using  Theorem \ref{t1} that
\begin{align*}
|\overline{\mathtt{A}_n}(f)(x)-f(x)|&=\left\vert \int_{-\infty}^{\infty}  \sum_{s=1}^r w_s f \left( \frac{v}{n}+\frac{s}{nr}\right) \Psi(nx-v) dv-  \int_{-\infty}^{\infty} \left( \sum_{s=1}^r w_s f(x)\right) \Psi(nx-v) dv\right\vert\\
&\leq \int_{-\infty}^{ \infty}  \sum_{s=1}^r w_s\left\vert f \left( \frac{v}{n}+\frac{s}{nr}\right)-f(x)\right\vert  \Psi(nx-v) dv\\
&\leq \int_{F_1}  \sum_{s=1}^r w_s \left\vert f \left( \frac{v}{n}+\frac{s}{nr}\right)-f(x)\right\vert  \Psi(nx-v) dv \\
&+\int_{F_2}  \sum_{s=1}^r w_s  \left\vert f \left( \frac{v}{n}+\frac{s}{nr}\right)-f(x)\right\vert  \Psi(nx-v) dv\\
&\leq \omega \left(f, \frac{1}{n}+\frac{1}{n^{\alpha}}\right)+ \frac{2 \left( q+\frac{1}{q} \right)\Vert f \Vert_{\infty}}{B^{\beta (n^{1-\alpha}-1)}}
\end{align*}
which gives the desired results. 
\end{proof}
\begin{proposition}
\begin{equation*}
\int_{-\infty}^{\infty} |h|^k \Psi(h)dh \leq \frac{B^{ \beta }-1}{(B^{ \beta }+1)(k+1)}  +\left( q+\frac{1}{q} \right)\frac{ B^{ \beta }}{ \beta ^k (\ln B)^k} \Gamma(k+1)<\infty
\end{equation*}
holds for $k \in \mathbb{N}$.
\end{proposition}
\begin{proof} 
We can easily compute that
\begin{align*}
\int_{-\infty}^{\infty} |h|^k \Psi(h)dh&=2\int_{0}^{\infty} h^k  \Psi(h)dh\\
&=2\left(\int_{0}^{1} h^k  \Psi(h)dh +\int_{1}^{\infty} h^k  \Psi(h)dh\right)\\
&\leq 2\left( \frac{B^{\beta }-1}{2(B^{\beta }+1)}\int_{0}^{1} h^k dh +\int_{1}^{\infty} h^k \frac{\frac{1}{2}\left( q+\frac{1}{q} \right) \beta  (\ln B)}{B^{\beta(h-1)}} dh\right)\\
&\leq  \frac{B^{ \beta }-1}{(B^{\beta }+1)(k+1)}  +\left( q+\frac{1}{q} \right)\beta  B^{\beta }(\ln B)\int_{0}^{\infty} h^k B^{-\beta  h} dh\\
&\leq \frac{B^{\beta }-1}{(B^{\beta }+1)(k+1)} +\left( q+\frac{1}{q} \right)\frac{ B^{\beta }}{ \beta ^k (\ln B)^k}\int_{0}^{\infty} x^k e^{-x} dx\\
&\leq \frac{B^{\beta }-1}{(B^{\beta }+1)(k+1)}  +\left( q+\frac{1}{q} \right)\frac{ B^{\beta }}{ \beta^k (\ln B)^k} \Gamma(k+1)< \infty\\
&\leq \frac{B^{\beta }-1}{(B^{\beta }+1)(k+1)}  +\left( q+\frac{1}{q} \right)\frac{ B^{\beta } k!}{ \beta ^k (\ln B)^k} < \infty.
\end{align*}
\end{proof}
\begin{theorem}
\label{theorem6}
\begin{equation}
\label{38}
\omega(\mathtt{A}_n(f),\theta)\leq\omega(f,\theta),~\theta >0
\end{equation}
 holds for $f\in C_B(\mathbb{R})\cup C_U(\mathbb{R})$. 
Furthermore, we have  $\mathtt{A}_n(f)\in C_U (\mathbb{R})$  for  $f\in C_U(\mathbb{R})$ where  $ C_U(\mathbb{R})$ is the set of all uniformly continuous functions on  $\mathbb{R}$.
\end{theorem}
\begin{proof}
Since 
\begin{equation*}
\mathtt{A}_n (f)(x)=\int_{-\infty}^{\infty} f\left( x-\frac{h}{n}\right)\Psi(h)dh,
\end{equation*}
let $x,y \in \mathbb{R}$ and then we can write that 
\begin{equation*}
\mathtt{A}_n (f)(x)-\mathtt{A}_n (f)(y)=\int_{-\infty}^{\infty} \left( f \left( x-\frac{h}{n}\right)-f\left( y-\frac{h}{n} \right)\right)\Psi(h)dh.
\end{equation*}
Hence
\begin{align*}
\left \vert \mathtt{A}_n (f)(x)-\mathtt{A}_n (f)(y)\right\vert &\leq \int_{-\infty}^{\infty} \left \vert f \left( x-\frac{h}{n}\right)-f\left( y-\frac{h}{n} \right)\right \vert \Psi(h)dh\\
& \leq \omega (f,|x-y|)\int_{-\infty}^{\infty}  \Psi(h)dh=\omega(f,|x-y|)
\end{align*} holds and by letting $|x-y|\leq \theta$, $\theta>0 $ then we obtain the desired results. 
\end{proof}
\begin{remark}
It is clear that the equality in  (\ref{38}) holds for $f=identity~map=:id$, and we have that 
\begin{equation*}
|\mathtt{A}_n(id)(x)-\mathtt{A}_n(id)(y)|=|id(x)-id(y)|=|x-y|, 
\end{equation*}
\begin{equation*}
\omega(\mathtt{A}_n(id), \theta)=\omega(id, \theta)=\theta>0. 
\end{equation*}
Also, we have that 
\begin{align*}
\mathtt{A}_n(id)(x)&=\int_{-\infty}^{\infty} \left( x-\frac{h}{n}\right) \Psi(h)dh \\
&=x\int_{-\infty}^{\infty}\Psi(h)dh-\frac{1}{n}\int_{-\infty}^{\infty} h \Psi(h)dh=x-\frac{1}{n}\int_{-\infty}^{\infty} h \Psi(h)dh
\end{align*}
and for fixed $x \in \mathbb{R}$,
\begin{align*}
|\mathtt{A}_n(id)(x)|&\leq |x|+\frac{1}{n} \int_{-\infty}^{\infty} |h| \Psi(h)dh \\
&=|x|+\frac{1}{n} \left[\frac{B^{\beta }-1}{2(B^{\beta }+1)} +\frac{ \left( q +\frac{1}{q}\right)B^{\beta }}{\beta \ln B}\right]< \infty.
\end{align*}
It is obvious that $id \in C_U (\mathbb{R}).$
\end{remark}
\begin{theorem}
\label{theorem7}
\begin{equation}
\label{42}
\omega (\mathtt{A}^*_n(f),\theta)\leq\omega(f,\theta),~\theta>0
\end{equation} holds for $f\in C_B(\mathbb{R})\cup C_U(\mathbb{R})$.
Furthermore, we have $\mathtt{A}^*_n(f)\in C_U(\mathbb{R})$ for $f\in C_U(\mathbb{R})$.
\end{theorem}
\begin{proof}
Notice that
\begin{equation*}
\mathtt{A}^*_n(f)(x)=n \int_{-\infty}^{ \infty} \left( \int^{\frac{1}{n}}_0 f \left( t+\left(x-\frac{h}{n}\right)\right)dt\right) \Psi(h)dh
\end{equation*}
 and
\begin{equation*}
\mathtt{A}^*_n(f)(y)=n \int_{-\infty}^{ \infty} \left( \int^{\frac{1}{n}}_0 f \left( t+\left(y-\frac{h}{n}\right)\right)dt\right)\Psi(h)dh
\end{equation*}
for $ x,~y \in \mathbb{R} $.\\
Hence $\left\vert \mathtt{A}^*_n (f)(x)-\mathtt{A}^*_n (f)(y)\right\vert$
\begin{align*}
\leq& n \int_{-\infty}^{\infty} \left( \int^{\frac{1}{n}}_0 \left \vert f \left( t+\left(x-\frac{h}{n}\right)\right)- f \left( t+\left(y-\frac{h}{n}\right)\right) \right \vert dt\right) \Psi(h)dh \\
\leq &\omega (f,|x-y|)\int_{-\infty}^{\infty} \Psi(h)dh=\omega(f,|x-y|)
\end{align*}
 and again implies  the desired results. 
Furthermore, 
\begin{equation*}
|\mathtt{A}^*_n(id)(x)-\mathtt{A}^*_n(id)(y)|=|x-y|
\end{equation*}
and 
\begin{align*}
|\mathtt{A}^*_n(id)(x)|&\leq n \int_{-\infty}^{ \infty} \left( \int^{\frac{1}{n}}_0 \left(|t|+|x|+\frac{|h|}{n}\right)dt\right) \Psi(h)dh \\& \leq \int_{-\infty}^{\infty} \left(\frac{1}{n}+|x|+\frac{|h|}{n}\right) \Psi(h)dh\\
&= \frac{1}{n}+|x|+\frac{1}{n} \int_{-\infty}^{\infty} |h|\Psi(h)dh < \infty.
\end{align*}
This proves the attainability of (\ref{42}).
\end{proof}
\begin{theorem}
\label{theorem8}
\begin{equation}
\label{46}
\omega( \overline{\mathtt{A}_n}(f),\theta)\leq\omega(f,\theta),~\theta>0
\end{equation}
holds for $f\in C_B(\mathbb{R}) \cup C_U(\mathbb{R}) $.
Furthermore, we have  $ \overline{A_n}(f)\in C_U(\mathbb{R})$ for $f\in C_U(\mathbb{R})$.
\end{theorem} 
\begin{proof}
Notice that
\begin{equation*}
\overline{\mathtt{A}_n}(f)(x)=\int_{-\infty}^{\infty} \left( \sum_{s=1}^r w_s f \left(\left(x- \frac{h}{n} \right)+\frac{s}{nr}\right)\right) \Psi(h)dh
\end{equation*}
 and
\begin{equation*}
\overline{\mathtt{A}_n}(f)(y)=\int_{-\infty}^{\infty} \left( \sum_{s=1}^r w_s f \left(\left(y- \frac{h}{n} \right)+\frac{s}{nr}\right)\right)\Psi(h)dh
\end{equation*}
for $x,~y \in \mathbb{R}$. \\
Hence  $\left\vert \overline{\mathtt{A}_n} (f)(x)-\overline{\mathtt{A}_n} (f)(y)\right\vert$
\begin{align*}
=&\left \vert \int_{-\infty}^{ \infty}  \sum_{s=1}^r w_s \left[ f \left(\left(x- \frac{h}{n} \right)+\frac{s}{nr}\right)-f \left(\left(y- \frac{h}{n} \right)+\frac{s}{nr}\right)\right] \Psi(h)dh \right \vert\\
\leq &  \int_{-\infty}^{ \infty}  \sum_{s=1}^r w_s\left \vert f \left(\left(x- \frac{h}{n} \right)+\frac{s}{nr}\right)-f \left(\left(y- \frac{h}{n} \right)+\frac{s}{nr}\right)\right \vert \Psi(h)dh \\
\leq &\omega (f,|x-y|)\int_{-\infty}^{\infty}\Psi(h)dh=\omega(f,|x-y|)
\end{align*}
\end{proof} and this implies the desired results.
Furthermore,
\begin{equation*}
|\overline{\mathtt{A}_n}(id)(x)-\overline{\mathtt{A}_n}(id)(y)|=|x-y|
\end{equation*}
and 
\begin{align*}
|\overline{\mathtt{A}_n}(f)(id)(x)|&\leq \int_{-\infty}^{ \infty} \left( \sum_{s=1}^r w_s f \left(|x|+ \frac{|h|}{n}+ \frac{1}{n} \right)\right)\Psi(h)dh\\
&\leq \int_{-\infty}^{ \infty} \left( |x|+ \frac{|h|}{n}+ \frac{1}{n}\right)\Psi(h)dh= \frac{1}{n}+|x|+\frac{1}{n}\int_{-\infty}^{\infty} |h|dh< \infty
\end{align*}
proving the attainability of (\ref{46}).

\begin{remark}
Fix $s \in \mathbb{N}$ and let  $f \in C^{(s)}(\mathbb{R})$ such that $f^{(k)} \in C_B (\mathbb{R})$ for $k=1,2,\ldots, s$. Then we can write that 
\begin{equation*}
\mathtt{A}_n(f)(x)=\int_{-\infty}^{ \infty} f\left(x-\frac{h}{n}\right) \Psi(h)dh
\end{equation*}
and with the use of  Leibniz's rule 
\begin{align*}
\frac{\partial^s \mathtt{A}_n(f)(x)}{\partial x^s} & =\int_{-\infty}^{\infty} f^{(s)}\left(x-\frac{h}{n}\right) \Psi(h)dh\\
&=\int_{-\infty}^{ \infty} f^{(s)}\left(\frac{v}{n}\right) \Psi(nx-v)dv=\mathtt{A}_n( f^{(s)})(x),~ for~all~x \in \mathbb{R}.
\end{align*}
\end{remark}
Clearly the followings also hold:
\begin{equation*}
\frac{\partial^s \mathtt{A}^*_n(f)(x)}{\partial x^s}=\mathtt{A}^*_n( f^{(s)})(x);~\frac{\partial^s \overline{ \mathtt{A}_n} (f)(x)}{\partial x^s}= \overline{ \mathtt{A}_n}( f^{(s)})(x).
\end{equation*}
\begin{remark}
Fix $s\in \mathbb{N}$  and let  $f \in C^{(s)} (\mathbb{R})$ such that  $f^{(k)} \in C_B (\mathbb {R})$ for $k=1,2,\ldots, s $. We  derive that  
\begin{align*}
\left( \mathtt{A}_n(f)\right)^{(k)}(x)&= \mathtt{A}_n(f^{(k)})(x),\\
\left(\mathtt{A}^*_n(f)\right)^{(k)}(x)&= \mathtt{A}^*_n(f^{(k)})(x),\\
\left(\overline{ \mathtt{A}_n}(f)\right)^{(k)}(x)&=\overline{ \mathtt{A}_n}(f^{(k)})(x),~ for~all~x\in \mathbb{R} ~and~ k=1,2,\ldots, s.
\end{align*}
\end{remark}
\noindent Then we have the followings under the assumptions of  $0<\alpha<1$ and $n \in \mathbb{N}$ such that $n^{1-\alpha}>2$.
\begin{theorem}
Let  $f^{(k)} \in C_B(\mathbb{R})$ for $k=1, 2, \cdots, s \in \mathbb{N}$. Then we have the followings:
\begin{itemize}
\item[(i)]
\begin{equation*}
\left| ( \mathtt{A}_n (f))^{(k)}(x)-f^{(k)}(x)\right|\leq\omega\left(f^{(k)},\frac{1}{n^{\alpha}}\right)+ \frac{2\left( q+\frac{1}{q} \right)  \Vert f^{(k)} \Vert_{\infty}}{B^{ \beta (n^{1-\alpha}-1)}}=:\mathtt{T}_k
\end{equation*}
and 
\begin{equation*}
\Vert (\mathtt{A}_n(f))^{(k)}-f^{(k)}\Vert_{\infty} \leq  \mathtt{T}_k,
\end{equation*}
\item[(ii)]
\begin{equation*}
\left\vert (\mathtt{A}_n^* (f))^{(k)}(x)-f^{(k)}(x)\right\vert \leq\omega\left(f^{(k)},\frac{1}{n}+\frac{1}{n^{\alpha}}\right)+ \frac{2\left( q+\frac{1}{q} \right)\Vert f^{(k)} \Vert_{\infty}}{B^{\beta (n^{1-\alpha}-1)}}=: \mathtt{E}_k
\end{equation*} 
and 
\begin{equation*}
\Vert (\mathtt{A}_n^*(f))^{(k)}-f^{(k)} \Vert_{\infty} \leq \mathtt{E}_k,
\end{equation*}
\begin{equation*}
\end{equation*}
\item[(iii)]
\begin{equation*}
\left\vert (\overline{\mathtt{A}_n} (f))^{(k)} (x)-f^{(k)}(x)\right\vert \leq\omega\left(f^{(k)},\frac{1}{n}+\frac{1}{n^{\alpha}}\right)+ \frac{ 2\left( q+\frac{1}{q} \right) \Vert f^{(k)} \Vert_{\infty}}{B^{\beta (n^{1-\alpha}-1)}}=\mathtt{E}_k
\end{equation*} 
and 
\begin{equation*}
\Vert (\overline{\mathtt{A}_n}(f))^{(k)}-f^{(k)}\Vert_{\infty} \leq \mathtt{E}_k.
\end{equation*}
\end{itemize}
\end{theorem}
\begin{proof}
Taking into consideration of Theorems \ref{theorem3}, \ref{theorem4}, \ref{theorem5}, we immediately obtain the proof.
\end{proof}
Now, we present a similar result on the preservation of global smoothness.
\begin{theorem}
Let $f^{(k)} \in C_B(\mathbb{R}) \cup C_U (\mathbb{R})$, for $k=0, 1, \cdots, s \in \mathbb{N}$. Then we have the followings:
\begin{itemize}
\item[(i)]
\begin{equation*}
\omega \left( (\mathtt{A}_n(f))^{(k)}, \theta \right)\leq \omega \left( f^{(k)}, \theta \right), ~\theta>0,
\end{equation*}
also $\left( \mathtt{A}_n(f) \right)^{(k)} \in C_U(\mathbb{R})$ when   $f^{(k)}\in C_U(\mathbb{R})$,
\item[(ii)]
\begin{equation*}
\omega \left( (\mathtt{A}^*_n(f))^{(k)}, \theta \right)\leq \omega \left( f^{(k)}, \theta \right), ~\theta>0,
\end{equation*}
also $\left( \mathtt{A}^*_n(f) \right)^{(k)} \in C_U(\mathbb{R})$ when   $f^{(k)}\in C_U(\mathbb{R})$,
\item[(iii)]
\begin{equation*}
\omega \left( (\overline{\mathtt{A}_n} (f))^{(k)}, \theta \right)\leq \omega \left( f^{(k)}, \theta \right), ~\theta>0,
\end{equation*}
also   $\left( \overline{\mathtt{A}_n}(f) \right)^{(k)} \in C_U(\mathbb{R})$ when   $f^{(k)}\in C_U(\mathbb{R})$.
\end{itemize}
\end{theorem}
\begin{proof}
Taking into consideration of Theorems \ref{theorem6}, \ref{theorem7}, \ref{theorem8}, we immediately obtain the proof.
\end{proof}
Now, we improve our results dealing with the improvement of the  of the rate of convergence of our operators by assuming the differentiability of functions.
\begin{theorem}
\label{10}
If  $0<\alpha<1$, $n \in \mathbb{N}:n^{1-\alpha}>2$, $x \in \mathbb{R}$, $f \in C^N (\mathbb{R})$, $N \in \mathbb{N}$, such that $f^{(N)} \in C_B(\mathbb{R})$, then the followings hold:
\begin{itemize}
\item[(i)]
\begin{align*}
\left| ( \mathtt{A}_n (f))(x)-f(x) - \sum_{k=1}^N \frac{f^{(k)}(x)}{k!} \left(\mathtt{A}_n\left( (\cdot -x)^k \right)\right)(x)\right|
\end{align*}
\begin{equation*}
\leq \frac{\omega \left( f^{(N)}, \frac{1}{n^{\alpha}} \right)}{n^{\alpha N}N!}+\frac{2^{N+2} \Vert f^{(N)} \Vert_{\infty}B^{ \beta } \left( q+\frac{1}{q} \right)}{n^N \beta ^N} B^{ \frac{-\beta  n^{1-\alpha}}{2}}\rightarrow 0, ~as~ n\rightarrow  \infty,
\end{equation*}
\item[(ii)]
if $ f^{(k)}(x)=0$, $k=1, 2, \cdots, N$  then
\begin{equation*}
\vert \mathtt{A}_n(f)(x)-f(x)\vert \leq \frac{\omega \left( f^{(N)}, \frac{1}{n^{\alpha}} \right)}{n^{\alpha N}N!}+ \frac{2^{N+2} \\ \Vert f^{(N)} \Vert_{\infty}B^{\beta } \left( q+\frac{1}{q} \right)}{n^N \beta ^N} B^{ \frac{-\beta  n^{1-\alpha}}{2}},
\end{equation*}
\item[(iii)]
\begin{align*}
\vert \mathtt{A}_n(f)(x)-f(x)\vert& \leq \sum_{k=1}^N \frac{\left \vert f^{(k)}(x) \right \vert}{k!} \frac{1}{n^k} \left[ \frac{B^{\beta}-1}{(B^{\beta }+1)(k+1)}  +\frac{ B^{\beta } \left( q+\frac{1}{q} \right) k!}{\beta ^k (\ln B)^k} \right]\\
&+  \frac{\omega \left( f^{(N)}, \frac{1}{n^{\alpha}} \right)}{n^{\alpha N}N!}+ \frac{2^{N+2}  \Vert f^{(N)} \Vert_{\infty}B^{\beta } \left( q+\frac{1}{q} \right)}{n^N \beta^N} B^{ \frac{-\beta  n^{1-\alpha}}{2}},
\end{align*}
\item[(iv)]
\begin{align*}
\Vert \mathtt{A}_n(f)-f\Vert_{\infty}& \leq \sum_{k=1}^N \frac{\left \Vert f^{(k)} \right \Vert_{\infty}}{k!} \frac{1}{n^k} \left[ \frac{B^{\beta }-1}{(B^{\beta}+1)(k+1)}  +\frac{ B^{\beta } \left( q+\frac{1}{q} \right) k!}{\beta ^k (\ln B)^k} \right]\\
&+  \frac{\omega \left( f^{(N)}, \frac{1}{n^{\alpha}} \right)}{n^{\alpha N}N!}+ \frac{2^{N+2}  \Vert f^{(N)} \Vert_{\infty}B^{ \beta } \left( q+\frac{1}{q} \right)}{n^N \beta ^N} B^{ \frac{-\beta  n^{1-\alpha}}{2}}.
\end{align*}
\end{itemize}
\end{theorem}
\begin{proof}
Since it is already known that 
\begin{equation*}
f \left(  \frac{v}{n}\right)=\sum_{k=0}^{N} \frac{f^{(k)} (x)}{k!} \left( \frac{v}{n}-x \right)^k+\int_x^{\frac{v}{n}} \left( f^{(N)}(t)-f^{(N)}(x)\right)\frac{ \left(\frac{v}{n}-t \right)^{N-1}}{(N-1)!}dt
\end{equation*}
and 
\begin{align*}
f \left(  \frac{v}{n}\right) \Psi(nx-v)& =\sum_{k=0}^{N} \frac{f^{(k)} (x)}{k!}  \Psi(nx-v) \left( \frac{v}{n}-x \right)^k\\
&+ \Psi(nx-v)\int_x^{\frac{v}{n}} \left( f^{(N)}(t)-f^{(N)}(x)\right)\frac{ \left(\frac{v}{n}-t \right)^{N-1}}{(N-1)!}dt,
\end{align*}
 we can write that
\begin{align*}
\mathtt{A}_n(f)(x)&=\int_{-\infty}^{\infty} f \left(  \frac{v}{n}\right) \Psi(nx-v)dv \\
&=\sum_{k=0}^{N} \frac{f^{(k)} (x)}{k!} \int_{-\infty}^{\infty} \Psi(nx-v) \left( \frac{v}{n}-x \right)^k dv\\
&+\int_{-\infty}^{\infty} \Psi(nx-v) \left( \int_x^{\frac{v}{n}} \left( f^{(N)}(t)-f^{(N)}(x)\right)\frac{ \left(\frac{v}{n}-t \right)^{N-1}}{(N-1)!}dt\right) dv
\end{align*}
and
\begin{align*}
 \mathtt{A}_n(f)(x)-f(x)&=\sum_{k=1}^{N} \frac{f^{(k)} (x)}{k!} \int_{-\infty}^{\infty} \Psi(nx-v) \left( \frac{v}{n}-x \right)^k dv\\
&+\int_{-\infty}^{\infty} \Psi(nx-v) \left( \int_x^{\frac{v}{n}} \left( f^{(N)}(t)-f^{(N)}(x)\right)\frac{ \left(\frac{v}{n}-t \right)^{N-1}}{(N-1)!}dt\right) dv.
\end{align*}
Let 
\begin{equation*}
R_n(x):=\int_{-\infty}^{\infty} \Psi(nx-v) \left( \int_x^{\frac{v}{n}} \left( f^{(N)}(t)-f^{(N)}(x)\right)\frac{ \left(\frac{v}{n}-t \right)^{N-1}}{(N-1)!}dt\right) dv
\end{equation*}
and 
\begin{equation*}
\gamma(v):= \int_x^{\frac{v}{n}} \left( f^{(N)}(t)-f^{(N)}(x)\right)\frac{ \left(\frac{v}{n}-t \right)^{N-1}}{(N-1)!}dt.
\end{equation*}
Now, we have two cases: \begin{itemize} \item[1.] 
$  \left \vert \frac{v}{n}-x  \right \vert< \frac{1}{n^\alpha},$\\
\item[2.] $  \left \vert \frac{v}{n}-x  \right \vert \geq \frac{1}{n^\alpha}.$
\end{itemize} 
Let us start with Case 1.\\
Case 1(i): When $ \frac{v}{n}\geq x$ we have 
\begin{align*}
|\gamma(v)|&
 \leq \omega \left( f^{(N)}, \frac{1}{n^{\alpha}} \right) \left(\int_x^{\frac{v}{n}} \frac{ \left(\frac{v}{n}-t \right)^{N-1}}{(N-1)!}dt\right) \\
&=\omega \left( f^{(N)}, \frac{1}{n^{\alpha}} \right) \frac{ \left(\frac{v}{n}-x\right)^{N}}{N!}\leq \omega \left( f^{(N)}, \frac{1}{n^{\alpha}} \right) \frac{1}{n^{\alpha N}N!}.
\end{align*}
Case 1 (ii): When $ \frac{v}{n} < x$, we have  
\begin{align*}
|\gamma(v)|&=  
\omega \left( f^{(N)}, \frac{1}{n^{\alpha}} \right) \frac{ \left(x-\frac{v}{n}\right)^{N}}{N!}\leq \omega \left( f^{(N)}, \frac{1}{n^{\alpha}} \right) \frac{ 1}{n^{\alpha N}N!}.
\end{align*}
Therefore $ |\gamma(v)| \leq \omega \left( f^{(N)}, \frac{1}{n^{\alpha}} \right) \frac{ 1}{n^{\alpha N}N!}$ holds whenever  $  \left \vert \frac{v}{n}-x  \right \vert< \frac{1}{n^\alpha} $.
Then we have that
\begin{equation*}
\left \vert \int_{ \left \vert \frac{v}{n}-x  \right \vert< \frac{1}{n^\alpha} } \Psi(nx-v) \left( \int_x^{\frac{v}{n}} \left( f^{(N)}(t)-f^{(N)}(x)\right)\frac{ \left(\frac{v}{n}-t \right)^{N-1}}{(N-1)!}dt\right) dv \right \vert 
\end{equation*}
\begin{equation*}
\leq \int_{ \left \vert \frac{v}{n}-x  \right \vert< \frac{1}{n^\alpha} } \Psi(nx-v)  |\gamma(v)| dv \leq  \omega \left( f^{(N)}, \frac{1}{n^{\alpha}} \right) \frac{ 1}{n^{\alpha N}N!}.
\end{equation*}
Case 2: In this case, we can write that
\begin{align*}
&\left \vert \int_{ \left \vert \frac{v}{n}-x  \right \vert\geq \frac{1}{n^\alpha} } \Psi(nx-v) \left( \int_x^{\frac{v}{n}} \left( f^{(N)}(t)-f^{(N)}(x)\right)\frac{ \left(\frac{v}{n}-t \right)^{N-1}}{(N-1)!}dt\right) dv \right \vert \\
&\leq  \int_{ \left \vert \frac{v}{n}-x  \right \vert\geq \frac{1}{n^\alpha} } \Psi(nx-v)\left \vert \int_x^{\frac{v}{n}} \left( f^{(N)}(t)-f^{(N)}(x)\right)\frac{ \left(\frac{v}{n}-t \right)^{N-1}}{(N-1)!}dt\right \vert dv \\
& =\int_{ \left \vert \frac{v}{n}-x  \right \vert\geq \frac{1}{n^\alpha} } \Psi(nx-v) |\gamma(v)|:= \mathtt{P}.
\end{align*}
If $\frac{v}{n} \geq x$ then $ |\gamma(v)|\leq 2\left \Vert f^{(N)} \right \Vert_{\infty} \frac{ \left(\frac{v}{n}-x \right)^{N}}{N!}  $ and 
if  $\frac{v}{n} < x$ then $ |\gamma(v)|\leq 2\left \Vert f^{(N)} \right \Vert_{\infty} \frac{ \left(x-\frac{v}{n} \right)^{N}}{N!}  $. So by combining them, we can also write  $ |\gamma(v)|\leq 2\left \Vert f^{(N)} \right \Vert_{\infty} \frac{ \left \vert \frac{v}{n}-x \right \vert^{N}}{N!} $
and by using Theorem \ref{t1} that
\begin{align*}
\mathtt{P}&\leq   \frac{2\left \Vert f^{(N)} \right \Vert_{\infty}}{N!}  \int_{ \left \vert \frac{v}{n}-x  \right \vert\geq \frac{1}{n^\alpha} } \Psi(nx-v) \left \vert x-\frac{v}{n} \right \vert^N dv\\
&= \frac{2\left \Vert f^{(N)} \right \Vert_{\infty}}{n^N N!}  \int_{ \left \vert nx-v \right \vert \geq n^{1-\alpha }} \Psi(|nx-v|) \left \vert nx-v \right \vert^N dv\\
&\leq \frac{2\left \Vert f^{(N)} \right \Vert_{\infty}}{n^N N!} \left(q +\frac{1}{q} \right)\beta  (\ln B) \int_{n^{1-\alpha}}^{\infty} B^{-\beta (x-1)} x^N dx\\
&= \frac{2\left \Vert f^{(N)} \right \Vert_{\infty}}{n^N N!} \left(q +\frac{1}{q} \right) \frac{(\ln B) B^\beta }{\beta ^N} \int_{n^{1-\alpha}}^{\infty} B^{-\beta  x} (\beta  x)^N (\beta  dx)\\
& =\frac{2\left \Vert f^{(N)} \right \Vert_{\infty}}{n^N N!} \left(q +\frac{1}{q} \right) \frac{(\ln B) B^\beta }{\beta^N}  \int_{\beta  n^{1-\alpha}}^{\infty} B^{-t} t^N dt\\
&\leq \frac{2\left \Vert f^{(N)} \right \Vert_{\infty}}{n^N N!} \left(q +\frac{1}{q} \right) \frac{(\ln B) B^\beta}{\beta ^N}  2^N N!\int_{\beta  n^{1-\alpha}}^{\infty} B^{-t} B^{\frac{t}{2}}dt\\
&=\frac{2^{N+1}\left \Vert f^{(N)} \right \Vert_{\infty} \left(q +\frac{1}{q} \right)}{n^N } \frac{(\ln B) B^\beta }{\beta ^N} \int_{\beta  n^{1-\alpha}}^{\infty}  B^{-\frac{t}{2}} dt\\
&=\frac{2^{N+1}\left \Vert f^{(N)} \right \Vert_{\infty} \left(q +\frac{1}{q} \right)}{n^N}  \frac{(\ln B) B^\beta }{\beta ^N} (-2)\frac{B^{-\frac{t}{2}}}{(\ln B)} \bigg\vert^{\infty}_{\beta  n^{1-\alpha}}\\
&=\frac{2^{N+2}  \Vert f^{(N)} \Vert_{\infty}B^{\beta } \left( q+\frac{1}{q} \right)}{n^N \beta ^N} B^{ \frac{-\beta  n^{1-\alpha}}{2}}.
\end{align*}
Then we get that  
\begin{equation*}
\left \vert \int_{ \left \vert \frac{v}{n}-x  \right \vert\geq \frac{1}{n^\alpha} } \Psi(nx-v) \left( \int_x^{\frac{v}{n}} \left( f^{(N)}(t)-f^{(N)}(x)\right)\frac{ \left(\frac{v}{n}-t \right)^{N-1}}{(N-1)!}dt\right) dv \right \vert
\end{equation*}
\begin{equation*}
\leq \frac{2^{N+2} \left \Vert f^{(N)} \right \Vert_{\infty} \left( q+ \frac{1}{q}\right) B^{\beta } B^{-\beta  n^{ \frac{1-\alpha}{2}}}}{n^N \beta ^N}\rightarrow0,~n\rightarrow \infty.
\end{equation*}
Finally we conclude that, 
\begin{align*}
\left \vert R_n(x) \right \vert &\leq \left \vert \int_{ \left \vert \frac{v}{n}-x  \right \vert< \frac{1}{n^\alpha} } \Psi(nx-v) \left( \int_x^{\frac{v}{n}} \left( f^{(N)}(t)-f^{(N)}(x)\right)\frac{ \left(\frac{v}{n}-t \right)^{N-1}}{(N-1)!}dt\right) dv \right \vert \\
&+\left \vert \int_{ \left \vert \frac{v}{n}-x  \right \vert\geq \frac{1}{n^\alpha} } \Psi(nx-v) \left( \int_x^{\frac{v}{n}} \left( f^{(N)}(t)-f^{(N)}(x)\right)\frac{ \left(\frac{v}{n}-t \right)^{N-1}}{(N-1)!}dt\right) dv \right \vert \\
&\leq \omega \left( f^{(N)}, \frac{1}{n^{\alpha}} \right) \frac{ 1}{n^{\alpha N}N!}+ \frac{2^{N+2} \left \Vert f^{(N)} \right \Vert_{\infty} \left( q+ \frac{1}{q}\right) B^{\beta } B^{-\beta  n^{ \frac{1-\alpha}{2}}}}{n^N \beta ^N}\rightarrow0,~n\rightarrow \infty.
\end{align*}
Now letting $k=1, 2, \cdots, N$ and $h=nx-v$ then we obtain that
\begin{align*}
\left \vert \mathtt{A}_n \left( ( \cdot-x)^k\right)(x) \right \vert& =\left \vert  \int_{-\infty}^{\infty} \Psi(nx-v) \left( \frac{v}{n}-x \right)^k dv \right \vert\\
& \leq \int_{-\infty}^{\infty}  \left \vert \frac{v}{n}-x  \right \vert^k \Psi(nx-v)dv\\
& =\frac{1}{n^k} \int_{-\infty}^{\infty}  \left \vert nx-v \right \vert^k \Psi(nx-v)dv=\frac{1}{n^k} \int_{-\infty}^{\infty}  \left \vert h \right \vert^k \Psi(h)dh\\
&\leq \frac{1}{n^k} \left[  \frac{B^{\beta }-1}{(B^{\beta }+1)(k+1)}  +\frac{ B^{\beta } \left( q+\frac{1}{q} \right) k!}{\beta ^k (\ln B)^k}  \right]\rightarrow 0, ~as ~n\rightarrow \infty.
\end{align*}
Hence we complete the proof. 
\end{proof}
\begin{theorem}
\label{11}
If $0<\alpha<1$, $n \in \mathbb{N}:n^{1-\alpha}>2$, $x \in \mathbb{R}$, $f \in C^{N}(\mathbb{R})$, $N \in \mathbb{N}$ such that $f^{(N)} \in C_B(\mathbb{R})$, then the  followings hold:
\begin{itemize}
\item[(i)]
\begin{align*}
\left| \mathtt{A}^*_n (f)(x)-f(x) - \sum_{k=1}^N \frac{f^{(k)}(x)}{k!} \left(\mathtt{A}^*_n\left( (\cdot -x)^k \right) \right)(x)\right|
\end{align*}
\begin{align*}
&\leq  \omega \left(f^{(N)}, \frac{1}{n}+\frac{1}{n^{\alpha}}\right) \frac{\left(  \frac{1}{n}+\frac{1}{n^{\alpha}}  \right)^N}{N!}\\
&+ \frac{ 2^N\left \Vert f^{(N)} \right \Vert_{\infty}}{n^N N!} \left( q+\frac{1}{q} \right) B^{\beta }  \left[ 1+  \frac{ 2^{N+1} N!}{\beta ^{N}} \right]  B^{-\frac{\beta  n^{1-\alpha}}{2}}\rightarrow 0,~n\rightarrow \infty,
\end{align*}
\item[(ii)]
if $ f^{(k)}(x)=0$, $k=1, 2, \cdots, N$ then 
\begin{align*}
\vert \mathtt{A}^*_n(f)(x)-f(x)\vert &\leq  \omega \left(f^{(N)}, \frac{1}{n}+\frac{1}{n^{\alpha}}\right) \frac{\left(   \frac{1}{n}+\frac{1}{n^{\alpha}} \right)^N}{ N!}\\
&+  \frac{ 2^N\left \Vert f^{(N)} \right \Vert_{\infty}}{ n^N N!} \left( q+\frac{1}{q} \right) B^{\beta }  \left[ 1+  \frac{ 2^{N+1} N!}{\beta ^{N}} \right]  B^{-\frac{\beta  n^{1-\alpha}}{2}},
\end{align*}
\item[(iii)]
\begin{align*}
\vert \mathtt{A}^*_n(f)(x)-f(x)\vert& \leq \sum_{k=1}^N \frac{\left \vert f^{(k)}(x)\right \vert}{k!} \frac{2^{k-1}}{n^k} \left[1+  \frac{ B^{\beta }-1}{(B^{\beta }+1)(k+1)}  +\frac{ B^{\beta } \left( q+\frac{1}{q} \right) k!}{\beta ^k (\ln B)^k}  \right] \\
&+  \omega \left(f^{(N)}, \frac{1}{n}+\frac{1}{n^{\alpha}}\right) \frac{\left(   \frac{1}{n}+\frac{1}{n^{\alpha}} \right)^N}{N!}\\
&+ \frac{ 2^N\left \Vert f^{(N)} \right \Vert_{\infty}}{n^N N!} \left( q+\frac{1}{q} \right) B^{\beta }  \left[ 1+  \frac{ 2^{N+1} N!}{\beta ^{N}} \right]  B^{-\frac{\beta  n^{1-\alpha}}{2}},
\end{align*}
\item[(iv)]
\begin{align*}
\Vert \mathtt{A}^*_n(f)-f\Vert_{\infty}& \leq \sum_{k=1}^N \frac{\left\Vert f^{(k)}\right \Vert_{\infty}}{k!} \frac{2^{k-1}}{n^k} \left[1+  \frac{ B^{\beta }-1}{(B^{\beta }+1)(k+1)}  +\frac{ B^{\beta } \left( q+\frac{1}{q} \right) k!}{\beta ^k (\ln B)^k}  \right] \\
&+ \omega \left(f^{(N)}, \frac{1}{n}+\frac{1}{n^{\alpha}}\right) \frac{\left(   \frac{1}{n}+\frac{1}{n^{\alpha}} \right)^N}{N!}\\
&+ \frac{ 2^N\left \Vert f^{(N)} \right \Vert_{\infty}}{n^N N!} \left( q+\frac{1}{q} \right) B^{\beta }  \left[ 1+  \frac{ 2^{N+1} N!}{\beta ^{N}} \right]  B^{-\frac{\beta  n^{1-\alpha}}{2}}.
\end{align*}
\end{itemize}
\end{theorem}
\begin{proof}
Since it is already known that
\begin{equation*}
f \left(  t+\frac{v}{n}\right)=\sum_{k=0}^{N} \frac{f^{(k)} (x)}{k!} \left( t+\frac{v}{n}-x \right)^k
+\int_x^{t+\frac{v}{n}} \left( f^{(N)}(s)-f^{(N)}(x)\right)\frac{ \left(t+\frac{v}{n}-s \right)^{N-1}}{(N-1)!}ds
\end{equation*}
and
\begin{align*}
\int_{0}^{\frac{1}{n}}f \left(  t+\frac{v}{n}\right)dt& =\sum_{k=0}^{N} \frac{f^{(k)} (x)}{k!} \int_{0}^{\frac{1}{n}} \left( t+\frac{v}{n}-x \right)^kdt+\int_{0}^{\frac{1}{n}}\left(\int_x^{t+\frac{v}{n}} \left( f^{(N)}(s)-f^{(N)}(x)\right)\frac{ \left(t+\frac{v}{n}-s \right)^{N-1}}{(N-1)!}ds\right)dt,
\end{align*}
we can write that 
\begin{align*}
&n \int_{-\infty}^{\infty} \left( \int_{0}^{\frac{1}{n}}f \left( t+ \frac{v}{n}\right)dt \right) \Psi(nx-v)dv\\
&=\sum_{k=0}^{N} \frac{f^{(k)} (x)}{k!} n \int_{-\infty}^{\infty} \left( \int_{0}^{\frac{1}{n}}f \left( t+ \frac{v}{n}-x \right)^k dt \right) \Psi(nx-v)dv\\
&+n \int_{-\infty}^{\infty}  \left( \int_{0}^{\frac{1}{n}} \left(  \int_x^{t+\frac{v}{n}} \left( f^{(N)}(s)-f^{(N)}(x)\right)\frac{ \left(t+\frac{v}{n}-s \right)^{N-1}}{(N-1)!}ds\right)dt \right)\Psi(nx-v)dv
\end{align*}
and 
\begin{align*}
 \mathtt{A}^*_n(f)(x)-f(x)&=\sum_{k=1}^{N} \frac{f^{(k)} (x)}{k!} \mathtt{A}^*_n \left( (\cdot -x)^k \right)(x)+R_n(x).
\end{align*}
Here, let
\begin{equation*}
R_n(x):= n \int_{-\infty}^{\infty}  \left( \int_{0}^{\frac{1}{n}} \left(  \int_x^{t+\frac{v}{n}} \left( f^{(N)}(s)-f^{(N)}(x)\right)\frac{ \left(t+\frac{v}{n}-s \right)^{N-1}}{(N-1)!}ds\right)dt\right) \Psi(nx-v)dv
\end{equation*}
and 
\begin{equation*}
\gamma(v):= n\int_{0}^{\frac{1}{n}} \left(  \int_x^{t+\frac{v}{n}} \left( f^{(N)}(s)-f^{(N)}(x)\right)\frac{ \left(t+\frac{v}{n}-s \right)^{N-1}}{(N-1)!}ds\right)dt.
\end{equation*}
Now we have two cases : \begin{itemize}
\item[1.] $ \left \vert \frac{v}{n}-x  \right \vert < \frac{1}{n^\alpha}, $
\item[2.] $ \left \vert \frac{v}{n}-x  \right \vert \geq \frac{1}{n^\alpha}. $
\end{itemize} 
Let us start with Case 1. 
\begin{itemize}
\item[Case 1 (i):]
 When $ t+\frac{v}{n} \geq x $, we have
\begin{align*}
|\gamma(v)| &\leq n \int_{0}^{\frac{1}{n}} \left(  \int_x^{t+\frac{v}{n}} \left\vert f^{(N)}(s)-f^{(N)}(x)\right \vert \frac{ \left( t+\frac{v}{n}-s \right)^{N-1}}{(N-1)!}ds \right)dt \\
& \leq n \int_{0}^{ \frac{1}{n}} \omega \left( f^{(N)}, |t|+\left \vert \frac{v}{n}-x \right \vert \right) \left(  \int_x^{t+\frac{v}{n}}  \frac{ \left(t+\frac{v}{n}-s \right)^{N-1}}{(N-1)!}ds \right)dt\\
&\leq \omega \left( f^{(N)}, \frac{1}{n}+\frac{1}{n^{\alpha}} \right) n  \int_{0}^{\frac{1}{n}}  \frac{\left( |t|+\left \vert \frac{v}{n}-x \right \vert \right)^N }{N!} dt \\
&\leq \frac{\omega \left(f^{(N)}, \frac{1}{n}+\frac{1}{n^{\alpha}}\right)}{N!} \left( \frac{1}{n}+\frac{1}{n^{\alpha}}\right)^N.
\end{align*}
\item[Case 1 (ii):]
When $ t+\frac{v}{n} < x $, we have
\begin{equation*}
|\gamma(v)|= n \left \vert \int_{0}^{\frac{1}{n}} \left(  \int_x^{t+\frac{v}{n}} \left( f^{(N)}(s)-f^{(N)}(x)\right) \frac{ \left( \left(t+\frac{v}{n} \right)-s\right)^{N-1}}{(N-1)!}ds\right)dt \right \vert
\end{equation*}
\begin{align*}
&\leq n\int_{0}^{\frac{1}{n}} \left(  \int^x_{t+\frac{v}{n}} \left\vert f^{(N)}(s)-f^{(N)}(x)\right \vert \frac{ \left(s- \left(t+\frac{v}{n} \right)\right)^{N-1}}{(N-1)!}ds\right)dt\\
&\leq n \int_{0}^{\frac{1}{n}} \omega \left( f^{(N)}, |t|+\left \vert \frac{v}{n}-x \right \vert \right)\left(  \int^x_{t+\frac{v}{n}} \frac{ \left(s- \left(t+\frac{v}{n} \right)\right)^{N-1}}{(N-1)!}ds\right)dt\\
&\leq \omega \left( f^{(N)}, \frac{1}{n}+\frac{1}{n^{\alpha}} \right) n  \int_{0}^{\frac{1}{n}} \frac{ \left(x- \left(t+\frac{v}{n} \right)\right)^{N}}{N!}dt\\
& \leq \omega \left( f^{(N)}, \frac{1}{n}+\frac{1}{n^{\alpha}} \right) n  \int_{0}^{\frac{1}{n}} \frac{ \left(\frac{1}{n}+\frac{1}{n^{\alpha}} \right)^{N}}{N!}dt \\
&=\omega \left( f^{(N)}, \frac{1}{n}+\frac{1}{n^{\alpha}} \right)    \frac{ \left(\frac{1}{n}+\frac{1}{n^{\alpha}} \right)^{N}}{N!}.
\end{align*}
\end{itemize}
By considering the above inequalities, we find that
\begin{equation*}
|\gamma(v)|\leq \omega \left( f^{(N)}, \frac{1}{n}+\frac{1}{n^{\alpha}} \right)    \frac{ \left(\frac{1}{n}+\frac{1}{n^{\alpha}} \right)^{N}}{N!}.
\end{equation*}
Then we conclude that 
\begin{equation*}
\left \vert n \int^{\infty}_{\substack{-\infty \\ \left \vert \frac{v}{n}-x \right \vert<\frac{1}{n^{\alpha}}}}\left( \int_{0}^{\frac{1}{n}} \left(  \int_x^{t+\frac{v}{n}} \left( f^{(N)}(s)-f^{(N)}(x)\right)\frac{ \left(t+\frac{v}{n}-s \right)^{N-1}}{(N-1)!}ds\right)dt\right) \Psi(nx-v)dv \right \vert
\end{equation*}
\begin{equation*}
\leq \omega \left( f^{(N)}, \frac{1}{n}+\frac{1}{n^{\alpha}} \right)    \frac{ \left(\frac{1}{n}+\frac{1}{n^{\alpha}} \right)^{N}}{N!}.
\end{equation*}
\begin{itemize}
\item[Case 2:]In this case,  we can write that
\end{itemize}
\begin{equation*}
\left \vert  \int_{\left \vert \frac{v}{n}-x \right \vert\geq\frac{1}{n^{\alpha}} } n \left( \int_{0}^{\frac{1}{n}} \left(  \int_x^{t+\frac{v}{n}} \left( f^{(N)}(s)-f^{(N)}(x)\right)\frac{ \left(t+\frac{v}{n}-s \right)^{N-1}}{(N-1)!}ds\right)dt\right) \Psi(nx-v)dv \right \vert
\end{equation*}
\begin{equation*}
\leq  \int_{\left \vert \frac{v}{n}-x \right \vert\geq\frac{1}{n^{\alpha}} } \left \vert n \left(  \int_{0}^{\frac{1}{n}} \left(  \int_x^{t+\frac{v}{n}} \left( f^{(N)}(s)-f^{(N)}(x)\right)\frac{ \left(t+\frac{v}{n}-s \right)^{N-1}}{(N-1)!}ds\right)d t\right)  \right\vert \Psi(nx-v)dv =:\mathtt{\Gamma} 
\end{equation*}
and
\begin{equation*}
|\gamma(v)|\leq n  \int_{0}^{\frac{1}{n}}\left \vert    \int_x^{t+\frac{v}{n}} \left( f^{(N)}(s)-f^{(N)}(x)\right) \frac{ \left(t+\frac{v}{n}-s \right)^{N-1}}{(N-1)!}ds \right \vert dt.
\end{equation*}
Now, if $t+\frac{v}{n}\geq x$, then
\begin{align*}
|\gamma(v)|\leq & 2 \left \Vert f^{(N)} \right \Vert_{\infty} n \int_0^\frac{1}{n} \frac{ \left(t+\frac{v}{n}-x \right)^{N}}{N!}dt\\
\leq& 2 \left \Vert f^{(N)} \right \Vert_{\infty} n \int_0^\frac{1}{n} \frac{ \left( |t|+\left \vert \frac{v}{n}-x \right \vert \right)^{N}}{N!}dt\\
\leq & 2 \frac{\left \Vert f^{(N)} \right \Vert_{\infty}}{N!} \left( \frac{1}{n}+\left \vert \frac{v}{n}-x \right \vert \right)^{N},
\end{align*}
and if  $t+\frac{v}{n} < x$ then $|\gamma(v)|\leq  2 \left \Vert f^{(N)} \right \Vert_{\infty} \frac{\left( \frac{1}{n}+\left \vert \frac{v}{n}-x \right \vert \right)^{N}}{N!}.$
 
By considering the above inequalities we conclude that 
\begin{equation*}
|\gamma(v)|\leq 2 \frac{\left \Vert f^{(N)} \right \Vert_{\infty}}{N!}  \left( \frac{1}{n}+\left \vert \frac{v}{n}-x \right \vert \right)^{N}
\end{equation*}
and by using Theorem \ref{t1} 
\begin{align*}
 \mathtt{\Gamma}&\leq \left( \int_{\left \vert \frac{v}{n}-x \right \vert\geq\frac{1}{n^{\alpha}} }  \left( \frac{1}{n}+\left \vert \frac{v}{n}-x \right \vert \right)^{N} \Psi(nx-v)dv \right)  \frac{ 2\left \Vert f^{(N)} \right \Vert_{\infty}}{N!} \\ 
& \leq \frac{ 2^N \left \Vert f^{(N)} \right \Vert_{\infty}}{N!} \int_{|nx-v|\geq n^{1-\alpha}} \left( \frac{1}{n^N}+ \frac{\left \vert nx-v\right \vert^N}{n^N}\right)\Psi(nx-v)dv\\
&\leq  \frac{2^N \left \Vert f^{(N)} \right \Vert_{\infty}}{n^N N!} \int_{|nx-v|\geq n^{1-\alpha}} \left( 1+ \left \vert nx-v\right \vert^N \right)\Psi(|nx-v|)dv\\
 &\leq \frac{ 2^N\left \Vert f^{(N)} \right \Vert_{\infty}}{n^N N!} \int_F \frac{1}{2}\left( q+\frac{1}{q} \right) \beta (\ln B)B^{-\beta \left(|nx-v|-1\right)}(1+|nx-v|)^N dv, ~F=\{v \in \mathbb{R}: |nx-v|\geq n^{1-\alpha} \}\\
 & \leq \frac{ 2^N\left \Vert f^{(N)} \right \Vert_{\infty}}{ n^N N!} \frac{1}{2}\left( q+\frac{1}{q} \right) \beta  (\ln B)\int_K B^{-\beta  \left(|nx-v|-1\right)}(1+|nx-v|)^N dv \\
 &= \frac{ 2^N\left \Vert f^{(N)} \right \Vert_{\infty}}{ n^N N!} \frac{1}{2}\left( q+\frac{1}{q} \right) \beta  (\ln B) \int_{n^{1-\alpha}}^{\infty} B^{-\beta  (x-1)} \left( 1+x^N \right)dx\\
 &= \frac{ 2^N\left \Vert f^{(N)} \right \Vert_{\infty}}{ n^N N!} \left( q+\frac{1}{q} \right)B^{\beta } \beta  (\ln B) \int_{n^{1-\alpha}}^{\infty} B^{-\beta  x} \left( 1+x^N \right)dx\\
 &= \frac{ 2^N\left \Vert f^{(N)} \right \Vert_{\infty}}{ n^N N!}\left( q+\frac{1}{q} \right)B^{ \beta } \beta  (\ln B) \left[  \int_{n^{1-\alpha}}^{\infty} B^{-\beta  x} dx+  \int_{n^{1-\alpha}}^{\infty} B^{-\beta  x} x^N dx \right]\\
 &= \frac{ 2^N\left \Vert f^{(N)} \right \Vert_{\infty}}{ n^N N!} \left( q+\frac{1}{q} \right)B^{\beta } \beta  (\ln B) \left[ \frac{ B^{-\beta  x}}{ \beta  (\ln B)} \bigg \vert_{n^{1-\alpha}}^{\infty}+  \int_{n^{1-\alpha}}^{\infty} B^{-\beta  x} x^N dx \right]\\
 &=\frac{ 2^N\left \Vert f^{(N)} \right \Vert_{\infty}}{ n^N N!} \left( q+\frac{1}{q} \right)B^{\beta } \beta  (\ln B) \left[ \frac{ B^{-\beta  (n^{1-\alpha})}}{ \beta  (\ln B)}+  \int_{n^{1-\alpha}}^{\infty} B^{-\beta  x} x^N dx \right]=\mathtt{M}\\
\end{align*}
and
\begin{align*}
\int_{n^{1-\alpha}}^{\infty} B^{-\beta  x} x^N dx&=\frac{1}{\beta ^{N+1}} \int_{n^{1-\alpha}}^{\infty}  B^{-\beta  x} (\beta  x)^N (\beta  dx)\\
& =\frac{1}{\beta ^{N+1}} \int_{\beta  n^{1-\alpha}}^{\infty}  B^{-t} t^N dt\\
&  \leq \frac{1}{\beta ^{N+1}} \int_{\beta  n^{1-\alpha}}^{\infty}  B^{-t} 2^N N! B^{\frac{t}{2}} dt\\
&=\frac{ 2^N N!}{\beta ^{N+1}} \int_{\beta  n^{1-\alpha}}^{\infty}  B^{-\frac{t}{2}} dt\\
&= \frac{ 2^{N+1} N!}{\beta ^{N+1}} \frac{- B^{-\frac{t}{2}}}{\ln B} \bigg \vert_{\beta  n^{1-\alpha}}^{\infty}\\
&= \frac{ 2^{N+1} N!}{\beta ^{N+1}} \frac{ B^{-\frac{\beta  n^{1-\alpha}}{2}}}{\ln B}.
\end{align*}
Then we obtain 
\begin{equation*}
\mathtt{M}\leq \frac{ 2^N\left \Vert f^{(N)} \right \Vert_{\infty}}{n^N N!} \left( q+\frac{1}{q} \right)B^{\beta }  \left[ B^{-\beta  (n^{1-\alpha})}+  \frac{ 2^{N+1} N!}{\beta ^{N}}  B^{-\frac{\beta  n^{1-\alpha}}{2}} \right],
\end{equation*}
and since 
\begin{equation*}
\beta  n^{1-\alpha} >\frac{\beta  n^{1-\alpha}}{2}\Rightarrow -\beta  n^{1-\alpha}< \frac{\beta  n^{1-\alpha}}{2}\Rightarrow  B^{-\beta  (n^{1-\alpha})} < B^{-\frac{\beta a n^{1-\alpha}}{2}}
\end{equation*}
we can easily write that 
\begin{align*}
\mathtt{M}&\leq \frac{ 2^N\left \Vert f^{(N)} \right \Vert_{\infty}}{n^N N!} \left( q+\frac{1}{q} \right)B^{\beta }  \left[ B^{-\beta  (n^{1-\alpha})}+  \frac{ 2^{N+1} N!}{\beta ^{N}}  B^{-\frac{\beta  n^{1-\alpha}}{2}} \right]\\
& =\frac{ 2^N\left \Vert f^{(N)} \right \Vert_{\infty}}{n^N N!} \left( q+\frac{1}{q} \right)B^{\beta } B^{-\beta  (n^{1-\alpha})} \left[1 +  \frac{ 2^{N+1} N!}{\beta ^{N}}  \right].
\end{align*}
Hence, we find 
\begin{equation*}
\left \vert  \int_{\left \vert \frac{v}{n}-x \right \vert\geq\frac{1}{n^{\alpha}} } n \left( \int_{0}^{\frac{1}{n}} \left(  \int_x^{t+\frac{v}{n}} \left( f^{(N)}(s)-f^{(N)}(x)\right)\frac{ \left(t+\frac{v}{n}-s \right)^{N-1}}{(N-1)!}ds\right)dt\right) \Psi(nx-v)dv \right \vert
\end{equation*}
\begin{equation*}
\leq \frac{ 2^N\left \Vert f^{(N)} \right \Vert_{\infty}}{ n^N N!} \left( q+\frac{1}{q} \right)B^{\beta }  \left[ 1+  \frac{ 2^{N+1} N!}{\beta ^{N}} \right]  B^{-\frac{\beta  n^{1-\alpha}}{2}}\rightarrow 0,~ n\rightarrow \infty.
\end{equation*}
Also
\begin{align*}
\left \vert R_n (x) \right \vert &\leq \omega \left(f^{(N)}, \frac{1}{n}+\frac{1}{n^{\alpha}} \right) \frac{\left( \frac{1}{n}+\frac{1}{n^{\alpha}}\right)^N}{N!}\\
& + \frac{ 2^N\left \Vert f^{(N)} \right \Vert_{\infty}}{ n^N N!} \left( q+\frac{1}{q} \right) B^{\beta }  \left[ 1+  \frac{ 2^{N+1} N!}{\beta ^{N}} \right]  B^{-\frac{\beta  n^{1-\alpha}}{2}}\rightarrow 0,~n\rightarrow \infty.
\end{align*}
So we notice for $k=1, 2, \cdots, N$ that 
\begin{align*}
\left \vert \mathtt{A}^*_n \left((\cdot -x)^k \right)\right \vert &= \left \vert  n \int_{-\infty}^{\infty} \left( \int_{0}^{\frac{1}{n}} \left( t+ \frac{v}{n}-x \right)^k dt \right) \Psi(nx-v)dv\right \vert\\
&\leq   n \int_{-\infty}^{\infty} \left(  \int_{0}^{\frac{1}{n}} \left \vert t+ \frac{v}{n}-x \right \vert^k dt \right) \Psi(nx-v)dv\\
& \leq n \int_{-\infty}^{\infty} \left( \int_{0}^{\frac{1}{n}} \left( | t|+ \left \vert \frac{v}{n}-x \right \vert \right)^k dt \right) \Psi(nx-v)dv\\
&\leq  \int_{-\infty}^{\infty}  \left(  \frac{1}{n}+ \left \vert \frac{v}{n}-x \right \vert \right)^k  \Psi(nx-v)dv\\
& =\frac{1}{n^k} \left[  \int_{-\infty}^{\infty}  \left( 1+|nx-v|\right)^k \Psi(nx-v)dv\right]\\
& \leq \frac{2^{k-1}}{n^k} \left[1+  \int_{-\infty}^{\infty}   |nx-v|^k \Psi(nx-v)dv\right]\\
& \leq \frac{2^{k-1}}{n^k} \left[1+  \int_{-\infty}^{\infty}   |h|^k  \Psi(h)dh\right]\\
& \leq \frac{2^{k-1}}{n^k} \left[1+  \frac{ B^{\beta }-1}{(B^{\beta }+1)(k+1)}  +\frac{ B^{\beta } \left( q+\frac{1}{q} \right) k!}{\beta ^k (\ln B)^k}  \right]\rightarrow 0, ~as ~n\rightarrow \infty
\end{align*}
and we complete the proof.
\end{proof}
Our next result deals with  activated Quadrature operators.
\begin{theorem}
\label{12}
If $0<\alpha<1$, $n \in \mathbb{N}:n^{1-\alpha}>2$, $x \in \mathbb{R}$, $f \in C^{N}(\mathbb{R})$, $N \in \mathbb{N}$ with $f^{(N)} \in C_B(\mathbb{R})$. Then the followings hold: 
\begin{itemize}
\item[(i)]
\begin{align*}
\left|\overline{\mathtt{A}_n} (f)(x)-f(x) - \sum_{k=1}^N \frac{f^{(k)}(x)}{k!} \left(\overline{ \mathtt{A}_n} \left( (\cdot -x)^k \right) \right)(x)\right| 
\end{align*}
\begin{align*}
&\leq  \omega \left(f^{(N)}, \frac{1}{n}+\frac{1}{n^{\alpha}}\right) \frac{\left(  \frac{1}{n}+\frac{1}{n^{\alpha}}  \right)^N}{N!}\\
&+\frac{ 2^N\left \Vert f^{(N)} \right \Vert_{\infty}}{ n^N N!} \left( q+\frac{1}{q} \right) B^{\beta }  \left[ 1+  \frac{ 2^{N+1} N!}{\beta ^{N}} \right]  B^{-\frac{\beta  n^{1-\alpha}}{2}}\rightarrow 0,~as~ n\rightarrow \infty,
\end{align*}
\item[(ii)]
if $ f^{(k)}(x)=0$, $k=1, 2, \cdots, N$ then 
\begin{align*}
\vert \overline{\mathtt{A}_n}f)(x)-f(x)\vert &\leq  \omega \left(f, \frac{1}{n}+\frac{1}{n^{\alpha}}\right) \frac{\left(   \frac{1}{n}+\frac{1}{n^{\alpha}} \right)^N}{N!}\\
&+\frac{ 2^N\left \Vert f^{(N)} \right \Vert_{\infty}}{ n^N N!} \left( q+\frac{1}{q} \right) B^{\beta }  \left[ 1+  \frac{ 2^{N+1} N!}{\beta ^{N}} \right]  B^{-\frac{\beta  n^{1-\alpha}}{2}},
\end{align*}
\item[(iii)]
\begin{align*}
\vert\overline{\mathtt{A}_n}(f)(x)-f(x)\vert& \leq \sum_{k=1}^N \frac{\left \vert f^{(k)}(x)\right \vert}{k!} \frac{2^{k-1}}{n^k} \left[1+  \frac{ B^{\beta }-1}{(B^{\beta }+1)(k+1)}  +\frac{ B^{\beta } \left( q+\frac{1}{q} \right) k!}{\beta ^k (\ln B)^k}  \right]\\
&+  \omega \left(f^{(N)}, \frac{1}{n}+\frac{1}{n^{\alpha}}\right) \frac{\left(   \frac{1}{n}+\frac{1}{n^{\alpha}} \right)^N}{N!}\\
&+ \frac{ 2^N\left \Vert f^{(N)} \right \Vert_{\infty}}{ n^N N!} \left( q+\frac{1}{q} \right) B^{\beta }  \left[ 1+  \frac{ 2^{N+1} N!}{\beta ^{N}} \right]  B^{-\frac{\beta  n^{1-\alpha}}{2}},
\end{align*}
\item[(iv)]
\begin{align*}
\Vert\overline{\mathtt{A}_n}(f)-f \Vert_{\infty}& \leq \sum_{k=1}^N \frac{\left\Vert f^{(k)} \right \Vert_{\infty}}{k!} \frac{2^{k-1}}{n^k} \left[1+  \frac{ B^{\beta }-1}{(B^{\beta }+1)(k+1)}  +\frac{ B^{\beta } \left( q+\frac{1}{q} \right) k!}{\beta ^k (\ln B)^k}  \right]  \\
&+ \omega \left(f^{(N)}, \frac{1}{n}+\frac{1}{n^{\alpha}}\right) \frac{\left(   \frac{1}{n}+\frac{1}{n^{\alpha}} \right)^N}{N!}\\
&+\frac{ 2^N\left \Vert f^{(N)} \right \Vert_{\infty}}{ n^N N!} \left( q+\frac{1}{q} \right) B^{\beta }  \left[ 1+  \frac{ 2^{N+1} N!}{\beta ^{N}} \right]  B^{-\frac{\beta  n^{1-\alpha}}{2}}.
\end{align*}
\end{itemize}
\end{theorem}
\begin{proof}
Since it is already known that 
\begin{align*}
f \left( \frac{v}{n}+\frac{s}{nr} \right)&=\sum_{k=0}^n \frac{f^{(k)}(x)}{k!}\left( \frac{v}{n}+\frac{s}{nr} -x \right)^k\\
&+ \int_{x}^{\frac{v}{n}+\frac{s}{nr}} \left( f^{(N)}(t)- f^{(N)}(x) \right) \frac{\left( \frac{v}{n}+\frac{s}{nr} -t \right)^{N-1}}{(N-1)!}dt
\end{align*}
and 
\begin{align*}
\sum_{s=1}^r w_s f \left( \frac{v}{n}+\frac{s}{nr} \right)&=\sum_{k=0}^n \frac{f^{(k)}(x)}{k!} \sum_{s=1}^r w_s\left( \frac{v}{n}+\frac{s}{nr} -x \right)^k\\
&+ \sum_{s=1}^r w_s \int_{x}^{\frac{v}{n}+\frac{s}{nr}} \left( f^{(N)}(t)- f^{(N)}(x) \right) \frac{\left( \frac{v}{n}+\frac{s}{nr} -t \right)^{N-1}}{(N-1)!}dt,
\end{align*}
we can write that  
\begin{align*}
\overline{\mathtt{A}_n}(f)(x)&= \int_{- \infty}^{\infty} \left( \sum_{s=1}^r w_s f \left( \frac{v}{n}+\frac{s}{nr} \right) \right)\Psi(nx-v)dv\\
&= \sum_{k=0}^n \frac{f^{(k)}(x)}{k!} \left(    \int_{- \infty}^{\infty} \sum_{s=1}^r w_s \left( \frac{v}{n}+\frac{s}{nr} -x \right)^k\           \right)\Psi(nx-v)dv\\
&+ \int_{- \infty}^{\infty} \left( \sum_{s=1}^r w_s \int_{x}^{\frac{v}{n}+\frac{s}{nr}} \left( f^{(N)}(t)- f^{(N)}(x) \right) \frac{\left( \frac{v}{n}+\frac{s}{nr} -t \right)^{N-1}}{(N-1)!} dt\right)\Psi(nx-v)dv,
\end{align*}
and 
\begin{align*}
\overline{\mathtt{A}_n}(f)(x)-f(x)& = \sum_{k=1}^N \frac{\left\Vert f^{(k)} \right \Vert_{\infty}}{k!} \left( \overline{\mathtt{A}_n} \left((\cdot -x)^k \right)(x)\right) +R_n(x), 
\end{align*}
where
\begin{equation*}
R_n(x):=\int_{- \infty}^{\infty} \left( \sum_{s=1}^r w_s \int_{x}^{\frac{v}{n}+\frac{s}{nr}}  \left( f^{(N)}(t)- f^{(N)}(x) \right) \frac{\left( \frac{v}{n}+\frac{s}{nr} -t \right)^{N-1}}{(N-1)!} dt\right)\Psi(nx-v)dv
\end{equation*}
and 
\begin{equation*}
\gamma(v):=\sum_{s=1}^r w_s \int_{x}^{\frac{v}{n}+\frac{s}{nr}} \left( f^{(N)}(t)- f^{(N)}(x) \right) \frac{\left( \frac{v}{n}+\frac{s}{nr} -t \right)^{N-1}}{(N-1)!}dt.
\end{equation*}
Now we have two cases: \begin{itemize}
\item[1.]  $ \left \vert \frac{v}{n}-x\right \vert< \frac{1}{n^{\alpha}},$
\item[2.]  $ \left \vert \frac{v}{n}-x\right \vert \geq \frac{1}{n^{\alpha}}.$
\end{itemize}
Let us start with Case 1. 
\begin{itemize}
\item[Case 1 (i):] When  $\frac{v}{n}+\frac{s}{nr}\geq x$, we have 
\begin{align*}
|\gamma(v)|&\leq\sum_{s=1}^r w_s \omega \left( f^{(N)}, \left \vert \frac{v}{n}+\frac{s}{nr} -x\right \vert \right) \frac{\left( \frac{v}{n}+\frac{s}{nr} -x \right)^{N}}{N!}dt\\
&\leq \omega \left(f^{(N)}, \frac{1}{n}+\frac{1}{n^{\alpha}}\right) \frac{\left(   \frac{1}{n}+\frac{1}{n^{\alpha}} \right)^N}{N!}.
\end{align*}
\item[Case 1 (ii):] When $\frac{v}{n}+\frac{s}{nr}<x$, we have
\begin{align*}
|\gamma(v)|&= \left \vert \sum_{s=1}^r w_s \int^{x}_{\frac{v}{n}+\frac{s}{nr}} \left( f^{(N)}(t)- f^{(N)}(x) \right) \frac{ \left( t-\left( \frac{v}{n}+\frac{s}{nr} \right) \right)^{N-1}}{(N-1)!}dt \right \vert \\
&\leq \sum_{s=1}^r w_s \int^{x}_{\frac{v}{n}+\frac{s}{nr}}  \left \vert f^{(N)}(t)- f^{(N)}(x) \right \vert \frac{ \left( t-\left( \frac{v}{n}+\frac{s}{nr} \right) \right)^{N-1}}{(N-1)!}dt\\
&\leq\sum_{s=1}^r w_s \omega \left( f^{(N)}, \left \vert x-\left( \frac{v}{n}+\frac{s}{nr} \right) \right \vert \right) \frac{\left( x-\left( \frac{v}{n}+\frac{s}{nr} \right)\right)^{N}}{N!}dt\\
&\leq \omega \left(f^{(N)}, \frac{1}{n}+\frac{1}{n^{\alpha}}\right) \frac{\left(   \frac{1}{n}+\frac{1}{n^{\alpha}} \right)^N}{N!}.
\end{align*}
\end{itemize}
By considering the above inequalities,  we find that 
\begin{equation*}
|\gamma(v)|\leq \omega \left(f^{(N)}, \frac{1}{n}+\frac{1}{n^{\alpha}}\right) \frac{\left(   \frac{1}{n}+\frac{1}{n^{\alpha}} \right)^N}{N!}
\end{equation*}
and
\begin{equation*}
\left \vert \int_{\left \vert \frac{v}{n}-x \right \vert< \frac{1}{n^{\alpha}}} \left( \sum_{s=1}^r w_s \int_{x}^{\frac{v}{n}+\frac{s}{nr}} \left( f^{(N)}(t)- f^{(N)}(x) \right) \frac{\left( \frac{v}{n}+\frac{s}{nr} -t \right)^{N-1}}{(N-1)!}dt \right)\Psi(nx-v)dv \right \vert
\end{equation*} 
\begin{equation*}
 \leq \int_{\left \vert \frac{v}{n}-x \right \vert< \frac{1}{n^{\alpha}}} \Psi(nx-v) |\gamma (v)|dv \leq \omega \left(f^{(N)}, \frac{1}{n}+\frac{1}{n^{\alpha}}\right) \frac{\left(   \frac{1}{n}+\frac{1}{n^{\alpha}} \right)^N}{N!}.
\end{equation*}
\begin{itemize}
\item[Case 2:] In this case, we can write that
\end{itemize}
\begin{equation*}
\left \vert \int_{\left \vert \frac{v}{n}-x \right \vert \geq \frac{1}{n^{\alpha}}} \Psi(nx-v)  \left( \sum_{s=1}^r w_s \int_{x}^{\frac{v}{n}+\frac{s}{nr}} \left( f^{(N)}(t)- f^{(N)}(x) \right) \frac{\left( \frac{v}{n}+\frac{s}{nr} -t \right)^{N-1}}{(N-1)!}dt \right)dv \right \vert
\end{equation*}
\begin{equation*}
 \leq \int_{\left \vert \frac{v}{n}-x \right \vert\geq \frac{1}{n^{\alpha}}} \Psi(nx-v) |\gamma (v)|dv =:\xi.
\end{equation*}
First let us consider the Case  $\frac{v}{n}+\frac{s}{nr} \geq x$. Then we have 
\begin{equation*}
 |\gamma(v)|\leq \frac{ 2 \left \Vert f^{(N)} \right \Vert_{\infty}}{N!} \sum_{s=1}^r w_s \left( \frac{v}{n}+\frac{s}{nr}-x \right)^N.
\end{equation*}
Now let us assume that  $\frac{v}{n}+\frac{s}{nr} < x$. Then we have 
\begin{align*}
|\gamma(v)|&= \left \vert \sum_{s=1}^r w_s \int^{x}_{\frac{v}{n}+\frac{s}{nr}} \left( f^{(N)}(t)- f^{(N)}(x) \right) \frac{ \left( t-\left( \frac{v}{n}+\frac{s}{nr} \right) \right)^{N-1}}{(N-1)!}dt \right \vert \\
&\leq \sum_{s=1}^r w_s \int^{x}_{\frac{v}{n}+\frac{s}{nr}}  \left \vert f^{(N)}(t)- f^{(N)}(x) \right \vert \frac{ \left( t-\left( \frac{v}{n}+\frac{s}{nr} \right) \right)^{N-1}}{(N-1)!}dt\\
&\leq  \frac{ 2 \left \Vert f^{(N)} \right \Vert_{\infty}}{N!} \sum_{s=1}^r w_s \left( x-\left(\frac{v}{n}+\frac{s}{nr} \right) \right)^N
\end{align*}
and consequently we can write for all cases 
\begin{equation*}
 |\gamma(v)|\leq \frac{ 2 \left \Vert f^{(N)} \right \Vert_{\infty}}{N!}\left( \left \vert x-\frac{v}{n}\right \vert+\frac{1}{n} \right)^N.
\end{equation*}
Similarly, as in the earlier theorem, we have that 
\begin{equation*}
 \left \vert  \int_{\left \vert \frac{v}{n}-x \right \vert \geq \frac{1}{n^{\alpha}}} \Psi(nx-v) \gamma (v)dv \right \vert\leq \xi
\end{equation*}
and
\begin{equation*}
\xi \leq \frac{ 2^N\left \Vert f^{(N)} \right \Vert_{\infty}}{ n^N N!} \left( q+\frac{1}{q} \right) B^{\beta }  \left[ 1+  \frac{ 2^{N+1} N!}{\beta ^{N}} \right]  B^{-\frac{\beta  n^{1-\alpha}}{2}}\rightarrow 0,~as~ n\rightarrow \infty.
\end{equation*}
We also see for $k=1, 2, \cdots, N$ that
\begin{align*}
\left\vert \overline{\mathtt{A}_n} \left((\cdot -x)^k \right)(x)\right \vert&=\left\vert   \int_{- \infty}^{\infty} \left( \sum_{s=1}^r w_s\left( \frac{v}{n}+\frac{s}{nr} -x \right)^k \right)\Psi(nx-v)dv \right \vert\\
&\leq \int_{- \infty}^{\infty} \left( \sum_{s=1}^r w_s \left\vert \frac{v}{n}+\frac{s}{nr} -x  \right \vert^k \right)\Psi(nx-v)dv\\
& \leq  \int_{- \infty}^{\infty} \left( \sum_{s=1}^r w_s \left(\left\vert  \frac{v}{n} -x  \right \vert +\frac{s}{nr} \right)^k \right)\Psi(nx-v)dv\\
&  \leq  \int_{- \infty}^{\infty} \left( \frac{1}{n}+\left\vert  \frac{v}{n} -x  \right \vert  \right)^k \Psi(nx-v)dv\\
& =\frac{1}{n^k} \int_{- \infty}^{\infty} \left( 1+\left\vert nx -v\right \vert  \right)^k \Psi(nx-v)dv\\
& \leq \frac{2^{k-1}}{n^k} \left[1+  \frac{ B^{\beta }-1}{(B^{\beta }+1)(k+1)}  +\frac{ B^{\beta } \left( q+\frac{1}{q} \right) k!}{\beta ^k (\ln B)^k}  \right]\rightarrow 0, ~as ~n\rightarrow \infty.
\end{align*}
Hence we complete the proof.
\end{proof}
\begin{theorem}
If $0<\alpha<1$, $n \in \mathbb{N}:n^{1-\alpha}>2$, $x \in \mathbb{R}$, $f^{(k)} \in C^{N}(\mathbb{R})$, $N \in \mathbb{N}$, $k=0, 1, \cdots, s \in \mathbb{N}$ with $f^{(N+k)} \in C_B(\mathbb{R})$. Then 
\begin{itemize}
\item[(i)]
\begin{align*}
\left| (\mathtt{A}_n (f))^{(k)} (x)-f^{(k)}(x) - \sum_{m=1}^N \frac{f^{(k+m)}(x)}{m!} \left( \mathtt{A}_n\left( (\cdot -x)^{m} \right) \right)(x) \right|
\end{align*}
\begin{equation*}
\leq \frac{\omega \left( f^{(N+k)}, \frac{1}{n^{\alpha}} \right)}{n^{\alpha N}N!}+\frac{2^{N+2} \Vert f^{(N+k)} \Vert_{\infty}B^{ \beta } \left( q+\frac{1}{q} \right)}{n^N \beta ^N} B^{ \frac{-\beta  n^{1-\alpha}}{2}},
\end{equation*}
\item[(ii)]
\begin{align*}
\left|( \mathtt{A}^*_n (f))^{(k)}(x)-f^{(k)}(x) - \sum_{m=1}^N \frac{f^{(k+m)}(x)}{m!} \left(\mathtt{A}^*_n\left( (\cdot -x)^m\right) \right)(x)\right|
\end{align*}
\begin{align*}
&\leq  \omega \left(f^{(N+k)}, \frac{1}{n}+\frac{1}{n^{\alpha}}\right) \frac{\left(  \frac{1}{n}+\frac{1}{n^{\alpha}}  \right)^N}{N!}\\
&+ \frac{ 2^N\left \Vert f^{(N+k)} \right \Vert_{\infty}}{n^N N!} \left( q+\frac{1}{q} \right) B^{\beta }  \left[ 1+  \frac{ 2^{N+1} N!}{\beta ^{N}} \right]  B^{-\frac{\beta  n^{1-\alpha}}{2}},
\end{align*}
\item[(iii)]
\begin{align*}
\left| (\overline{\mathtt{A}_n} (f))^{(k)}(x)-f^{(k)}(x) - \sum_{m=1}^N \frac{f^{(k+m)}(x)}{m!} \left( \overline{ \mathtt{A}_n} \left( (\cdot -x)^{m} \right) \right)(x)\right| 
\end{align*}
\begin{align*}
&\leq  \omega \left(f^{(N+k)}, \frac{1}{n}+\frac{1}{n^{\alpha}}\right) \frac{\left(  \frac{1}{n}+\frac{1}{n^{\alpha}}  \right)^N}{N!}\\
&+\frac{ 2^N\left \Vert f^{(N+k)} \right \Vert_{\infty}}{ n^N N!} \left( q+\frac{1}{q} \right) B^{\beta }  \left[ 1+  \frac{ 2^{N+1} N!}{\beta ^{N}} \right]  B^{-\frac{\beta  n^{1-\alpha}}{2}}
\end{align*}
\end{itemize}
\end{theorem}
\begin{proof}
By using Theorems \ref{10}, \ref{11}, \ref{12}, we immediately obtain the proof. 
\end{proof}

\section{Iterated Neural Network Operators of Convolution Type }
In this section, we consider the iterated versions of our operators  activated by symmetrized, deformed and parametrized $B-$ generalized logistic function with the motivation of \cite{G.A.2024}.
\begin{remark}(About Iterated Convolution)\\
Notice that 
\begin{equation*}
\mathtt{A}_n(f)(x)=\int_{-\infty}^{ \infty} f \left(x-\frac{h}{n}\right) \Psi(h)dh, ~for~f \in C_{B}(\mathbb{R})
\end{equation*}
and let $x_{k}\rightarrow x$, as $k\rightarrow \infty$, and 
\begin{equation*}
\mathtt{A}_n(f)(x_{k})-\mathtt{A}_n(f)(x)=\int_{-\infty}^{ \infty} 
\left[f \left(x_k-\frac{h}{n}\right)-f \left(x-\frac{h}{n}\right)\right] \Psi(h)dh.
\end{equation*}
We have that 
\begin{equation*}
f \left(x_k-\frac{h}{n}\right) \Psi(h)\rightarrow f \left(x-\frac{h}{n}\right)\Psi(h), ~for~every~ h \in \mathbb{R}, ~ k \rightarrow \infty.
\end{equation*}
\end{remark}
Furthermore
\begin{equation*}
|\mathtt{A}_n(f)(x_{k})-\mathtt{A}_n(f)(x)|
\leq \int_{-\infty}^{ \infty} 
\left \vert f \left(x_k-\frac{h}{n}\right)-f \left(x-\frac{h}{n}\right) \right \vert \Psi(h) dh\rightarrow 0, ~k \rightarrow \infty,
\end{equation*}
by Dominated Convergence Theorem, since 
\begin{equation*}
\left \vert f \left(x_k-\frac{h}{n}\right) \right \vert \Psi(h)dh \leq \Vert f \Vert_{\infty}  \Psi(h)
\end{equation*}
and $\Vert f \Vert_{\infty}  \Psi(h) $ is integrable over $\left( -\infty, \infty \right)$, for every $  h \in \left( -\infty, \infty \right)$.
Thus $\mathtt{A}_n(f) \in C_{B}(\mathbb{R})$.\\
 Also we have that
\begin{equation*}
|\mathtt{A}_n(f)(x)|\leq \Vert f \Vert _{\infty} \int_{-\infty}^{ \infty} \Psi(h)dh =\Vert f \Vert_{\infty}
\end{equation*}
i.e.,
\begin{equation*}
\Vert \mathtt{A}_n(f) \Vert _{\infty} \leq \Vert f \Vert _{\infty}
\end{equation*}
 which means that $A_n$ is bounded and linear for $n \in \mathbb{N}.$
\begin{remark} 
Let $r \in \mathbb{N}$. Since 
\begin{align*}
\mathtt{A}_n^r f-f&=(\mathtt{A}_n^r f-\mathtt{A}_n^{r-1} f)+(\mathtt{A}_n^{r-1} f+\mathtt{A}_n^{r-2} f)+(\mathtt{A}_n^{r-2} f+\mathtt{A}_n^{r-3} f)\\
&+\ldots + (\mathtt{A}_n^2 f-\mathtt{A}_n f)+(\mathtt{A}_n^r f-f).
\end{align*}
\end{remark} 
We have that  $\Vert \mathtt{A}_n^r f-f \Vert _{\infty} \leq r \Vert \mathtt{A}_n f-f\Vert _{\infty} $ and 
\begin{align*}
\mathtt{A}_{k_r}(\mathtt{A}_{k_{r-1}}(\ldots \mathtt{A}_{k_2}(\mathtt{A}_{k_1} f)))-f&=\ldots=\mathtt{A}_{k_r}(\mathtt{A}_{k_{r-1}}(\ldots \mathtt{A}_{k_2}))(\mathtt{A}_{k_1}-f)\\
&+\mathtt{A}_{k_r}(\mathtt{A}_{k_{r-1}}(\ldots \mathtt{A}_{k_3}))(\mathtt{A}_{k_2}-f)\\
&+\mathtt{A}_{k_r}(\mathtt{A}_{k_{r-1}}(\ldots \mathtt{A}_{k_4}))(\mathtt{A}_{k_3}-f)\mathtt{A}_{k_r}\\
&+\ldots+(\mathtt{A}_{k_{r-1}} f-f)+\mathtt{A}_{k_r}f-f
\end{align*}
where $k_1, k_2, \ldots, k_r\in \mathbb{N}:k_1\leq k_2\leq \ldots \leq k_r $ and
\begin{equation*}
\Vert \mathtt{A}_{k_r}(\mathtt{A}_{k_{r-1}}(\ldots \mathtt{A}_{k_2}(\mathtt{A}_{k_1} f)))-f\Vert_{\infty}\leq \sum_{m=1}^r\Vert \mathtt{A}_{k_m}f-f\Vert_{\infty}
\end{equation*}
as in  \cite{G.A.2024}. The similar results can be obtained for $\mathtt{A}^*_n$ and $ \overline{\mathtt{A}_n}$ using the same technique. 
\begin{remark}
Notice that 
\begin{equation*}
\mathtt{A}^*_n(f)(x)=n \int_{-\infty}^{ \infty} \left( \int^{\frac{1}{n}}_0 f \left( t+\left(x-\frac{h}{n}\right)\right)dt\right) \Psi(h) dh, ~f\in C_{B}(\mathbb{R}),
\end{equation*}
and let $x_{k}\rightarrow x$, as $k\rightarrow \infty$.
\end{remark}
and
\begin{equation*}
\left\vert \mathtt{A}^*_n (f)(x_k)-\mathtt{A}^*_n (f)(x)\right\vert 
\end{equation*}
\begin{align*}
\leq& n \int_{-\infty}^{ \infty} \left( \int^{\frac{1}{n}}_0 \left \vert f \left( t+\left(x_N-\frac{h}{n}\right)\right) - f \left( t+\left(x-\frac{h}{n}\right)\right) \right \vert dt\right) \Psi(h)dh,
\end{align*}
  by  Bounded Convergence Theorem, we get that:
\begin{equation*}
x_k\rightarrow x\rightsquigarrow t+\left(x_k-\frac{h}{n}\right)\rightarrow t+\left(x-\frac{h}{n}\right)
\end{equation*}
and 
\begin{equation*}
f \left( t+\left(x_k-\frac{h}{n}\right)\right)\rightarrow f \left( t+\left(x-\frac{h}{n}\right)\right), 
\end{equation*}
\begin{equation*}
\left \vert f \left( t+\left(x_k-\frac{h}{n}\right)\right)\right \vert \leq \Vert f \Vert_{\infty}
\end{equation*}
and $\left[ 0, \frac{1}{n} \right ] $ is finite. Hence
\begin{equation*}
n \int^{\frac{1}{n}}_0 \left \vert f \left( t+\left(x_k-\frac{h}{n}\right)\right) - f \left( t+\left(x-\frac{h}{n}\right)\right) \right \vert dt \rightarrow0, ~as~ k\rightarrow \infty.
\end{equation*}
Therefore it holds  that
\begin{equation*}
 n\int^{\frac{1}{n}}_0 f \left( t+\left(x_k-\frac{h}{n}\right)\right) dt\rightarrow n \int^{\frac{1}{n}}_0 f \left( t+\left(x-\frac{h}{n}\right)\right) dt \rightarrow0, ~as~ k\rightarrow \infty
\end{equation*}
and
\begin{equation*}
 \left( n\int^{\frac{1}{n}}_0 f \left( t+\left(x_k-\frac{h}{n}\right)\right) dt \right)\Psi(h)\rightarrow \left( n \int^{\frac{1}{n}}_0 f \left( t+\left(x-\frac{h}{n}\right)\right) dt \right)\Psi(h), 
\end{equation*}
as $k\rightarrow \infty, ~for~every~ h \in \left( -\infty, \infty \right)$.
Also, we get 
\begin{equation*}
 \left \vert \left( n\int^{\frac{1}{n}}_0 f \left( t+\left(x_k-\frac{h}{n}\right)\right) dt \right)\Psi(h) \right \vert \leq \Vert f \Vert_{\infty}\Psi(h).
\end{equation*}
Again by Dominated Convergence Theorem,
\begin{equation*}
\mathtt{A}^*_n (f)(x_k)\rightarrow \mathtt{A}^*_n (f)(x),~as~ k\rightarrow \infty.
\end{equation*}
Thus  $ \mathtt{A}^*_n (f)(x)$ is bounded and continuous in $x \in \left( -\infty, \infty \right)$ and the iterated facts hold for $\mathtt{A}^*_n$ as in the $ \mathtt{A}_n (f)$ case, all the same. \\
\begin{remark}
Next we observe that: $f\in C_{B}(\mathbb{R})$, and
\begin{equation}
 \overline{\mathtt{A}_n}(f)(x):=\int_{-\infty}^{\infty} \left( \sum_{s=1}^r w_s f \left( \left(x-\frac{h}{n}\right)+\frac{s}{nr}\right)\right) \Psi(h)dh.
\end{equation}
Let $x_{k}\rightarrow x$, as $k\rightarrow \infty$.
\end{remark}
 Then
\begin{align*}
\left \vert \overline{\mathtt{A}_n}(f)(x_k)-\overline{\mathtt{A}_n}(f)(x)\right \vert &=\left\vert \int_{-\infty}^{ \infty} \left(  \sum_{s=1}^r w_s \left( f \left( \left(x_k-\frac{h}{n}\right)+\frac{s}{nr}\right)- f \left( \left(x-\frac{h}{n}\right)+\frac{s}{nr}\right)\right) \right) \Psi(h)dh \right\vert\\
&\leq  \int_{-\infty}^{ \infty}  \left\vert \sum_{s=1}^r w_s \left( f \left( \left(x_k-\frac{h}{n}\right)+\frac{s}{nr}\right)- f \left( \left(x-\frac{h}{n}\right)+\frac{s}{nr} \right) \right)\right\vert \Psi(h)dh \rightarrow 0, ~as~ k\rightarrow \infty.
\end{align*}
The last is from Dominated Convergence Theorem:
\begin{equation*}
\left(x_k-\frac{h}{n}\right)+\frac{s}{nr}\rightarrow \left(x-\frac{h}{n}\right)+\frac{s}{nr}
\end{equation*}
and
\begin{equation*}
\sum_{s=1}^r w_s f \left( \left(x_k-\frac{h}{n}\right)+\frac{s}{nr}\right)\rightarrow \sum_{s=1}^r w_s f \left( \left(x-\frac{h}{n}\right)+\frac{s}{nr} \right)
\end{equation*}
and 
\begin{equation*}
\left(\sum_{s=1}^r w_s f \left( \left(x_k-\frac{h}{n}\right)+\frac{s}{nr}\right)\right)\Psi(h)\rightarrow \left(\sum_{s=1}^r w_s \left( \left(x-\frac{h}{n}\right)+\frac{s}{nr} \right)\right)\Psi(h),
\end{equation*}
as $k\rightarrow \infty, \forall h \in  \left( -\infty, \infty \right)$.
Furthermore, we have 
\begin{equation*}
 \left \vert \sum_{s=1}^r w_s f \left( \left(x_k-\frac{h}{n}\right)+\frac{s}{nr}\right) \right \vert \Psi(h)  \leq \Vert f \Vert_{\infty}\Psi(h)
\end{equation*}
which the last one function is integrable over $ \left( -\infty, \infty \right)$.\\
Therefore 
\begin{equation*}
\overline{\mathtt{A}_n}(f)(x_k)\rightarrow \overline{\mathtt{A}_n}(f)(x),~as~ k\rightarrow \infty.
\end{equation*}
Thus $ \overline{\mathtt{A}_n} (f)(x)$ is  bounded and continuous in $x \in \left( -\infty, \infty \right)$ and the iterated facts holds all the same.
\begin{theorem}
If  $0<\alpha<1$, $n \in \mathbb{N}:n^{1-\alpha}>2$, $r \in \mathbb{N}$, $f \in C_B(\mathbb{R})$, then we have the followings:
\begin{itemize}
\item[(i)] \begin{equation*}
\left \Vert \mathtt{A}_n^r (f) -f \right \Vert_{\infty} \leq r \left \Vert \mathtt{A}_n f -f \right \Vert_{\infty} \leq r \left[ \omega\left(f,\frac{1}{n^{\alpha}}\right)+ \frac{2\left( q+\frac{1}{q} \right)  \Vert f \Vert_{\infty}}{B^{\beta (n^{1-\alpha}-1)}} \right],
\end{equation*}
\item[(ii)] \begin{equation*}
\left \Vert \mathtt{A}_n^{*r} (f) -f \right \Vert_{\infty} \leq r \left \Vert \mathtt{A}^*_n f -f \right \Vert_{\infty} \leq r \left[ \omega\left(f,\frac{1}{n}+\frac{1}{n^{\alpha}}\right)+ \frac{2\left( q+\frac{1}{q} \right)  \Vert f \Vert_{\infty}}{B^{\beta (n^{1-\alpha}-1)}} \right],
\end{equation*}
\item[(iii)] \begin{equation*}
\left \Vert\overline{\mathtt{A}_n}^{r} (f) -f \right \Vert_{\infty} \leq r \left \Vert\overline{\mathtt{A}_n} f -f \right \Vert_{\infty} \leq r \left[ \omega\left(f,\frac{1}{n}+\frac{1}{n^{\alpha}}\right)+ \frac{2\left( q+\frac{1}{q} \right)  \Vert f \Vert_{\infty}}{B^{\beta (n^{1-\alpha}-1)}} \right].
\end{equation*}
\end{itemize}
The rate of convergence of $\mathtt{A}_n^r $, $\mathtt{A}_n^{*r} $, $\overline{\mathtt{A}_n}^{r}$ to identity  is not worse than the rate of convergence of  $\mathtt{A}_n $, $\mathtt{A}_n^{*} $, $\overline{\mathtt{A}_n}$.
\end{theorem}
\begin{proof}
The proof follows from Theorems  \ref{theorem3}, \ref{theorem4}, \ref{theorem5}.
\end{proof}

\begin{theorem}
If $0<\alpha<1$, $n \in \mathbb{N}; k_1, k_2, \cdots, k_r \in \mathbb{N}:k_1\leq k_2 \leq  \cdots \leq k_r $, $ k_p^{1-\alpha}>2$, $p=1, 2, \cdots, r;$  $f \in C_B(\mathbb{R})$,  then we have the followings:
\begin{itemize}
\item[(i)] \begin{align*}
\Vert \mathtt{A}_{k_r}(\mathtt{A}_{k_{r-1}}(\ldots \mathtt{A}_{k_2}(\mathtt{A}_{k_1} f)))-f\Vert_{\infty}&\leq \sum_{p=1}^r\Vert \mathtt{A}_{k_p}f-f\Vert_{\infty}\\
& \leq \sum_{p=1}^r  \left[ \omega\left(f,\frac{1}{k^{\alpha}_p}\right)+ \frac{2\left( q+\frac{1}{q} \right)  \Vert f \Vert_{\infty}}{B^{\beta (k_p^{1-\alpha}-1)}} \right] \\
&\leq r \left[ \omega\left(f,\frac{1}{k^{\alpha}_1}\right)+ \frac{2\left( q+\frac{1}{q} \right)  \Vert f \Vert_{\infty}}{B^{\beta (k_1^{1-\alpha}-1)}} \right],
\end{align*}
\item[(ii)] \begin{align*}
\Vert \mathtt{A}^*_{k_r}(\mathtt{A}^*_{k_{r-1}}(\ldots \mathtt{A}^*_{k_2}(\mathtt{A}^*_{k_1} f)))-f\Vert_{\infty}&\leq \sum_{p=1}^r\Vert \mathtt{A}^*_{k_p}f-f\Vert_{\infty}\\
& \leq \sum_{p=1}^r \left[ \omega\left(f,\frac{1}{k_p}+\frac{1}{k_p^{\alpha}}\right)+ \frac{2\left( q+\frac{1}{q} \right)  \Vert f \Vert_{\infty}}{B^{\beta (k_p^{1-\alpha}-1)}} \right] \\
&\leq r \left[ \omega\left(f,\frac{1}{k_1}+\frac{1}{k_1^{\alpha}}\right)+ \frac{2\left( q+\frac{1}{q} \right)  \Vert f \Vert_{\infty}}{B^{\beta (k_1^{1-\alpha}-1)}} \right],
\end{align*}
\item[(iii)] \begin{align*}
\Vert\overline{\mathtt{A}}_{k_r}(\mathtt{A}^*_{k_{r-1}}(\ldots \overline{\mathtt{A}}_{k_2}(\mathtt{A}^*_{k_1} f)))-f\Vert_{\infty}&\leq \sum_{p=1}^r\Vert \overline{\mathtt{A}}_{k_p}f-f\Vert_{\infty}\\
& \leq \sum_{p=1}^r \left[ \omega\left(f,\frac{1}{k_p}+\frac{1}{k_p^{\alpha}}\right)+ \frac{2\left( q+\frac{1}{q} \right)  \Vert f \Vert_{\infty}}{B^{\beta (k_p^{1-\alpha}-1)}} \right] \\
&\leq r \left[ \omega\left(f,\frac{1}{k_1}+\frac{1}{k_1^{\alpha}}\right)+ \frac{2\left( q+\frac{1}{q} \right)  \Vert f \Vert_{\infty}}{B^{ \beta (k_1^{1-\alpha}-1)}} \right],
\end{align*}
\end{itemize}
\end{theorem}
\begin{proof}
The proof follows from Theorems  \ref{theorem3}, \ref{theorem4}, \ref{theorem5}.
\end{proof}


\newpage

\begin{thebibliography}{99}
\bibitem{ALADAg} Aladağ, Ç.H., Eğrioğlu, E., Kadılar, C.,  Forecasting nonlinear time series with a hybrid methodology. Applied Mathematics Letters, 22, 1467-1470,  (2009).
\bibitem{G.A.2023} Anastassiou, G.A., General sigmoid based Banach space valued neural network approximation.
J. Comput. Anal. Appl. 31(4), 520–534, (2023).
\bibitem{G.A.2022} Anastassiou, G.A., $q$-Deformed and $\lambda$-parametrized $A$-generalized logistic function induced
Banach space valued multivariate multi layer neural network approximations. Stud. Univ. Babeş-Bolyai Math. 69, No. 3, 587-612, (2024).
\bibitem{G.A.2024}  Anastassiou, G.A., Approximation by parametrized logistic activated convolution type operators, Rev. Real Acad. Cienc. Exactas Fis. Nat. Ser. A-mat. 118:138, (2024).
\bibitem{Anastassiou20232} Anastassiou, G.A., Intelligent Computations: Parametrized. Deformed and General Neural Networks, Springer, Heidelberg, New York, (2023). 
\bibitem{anast 1997}Anastassiou, G.A.,  Rate of convergence of some neural network operators to the unit-univariate case. J. Math. Anal. Appl. 212, 237–262, (1997).
\bibitem{anast 2001}  Anastassiou, G.A.,  Quantitative Approximations Chapman and Hall/CRC, Boca Raton, New
York, (2001).
\bibitem{anast 2011} Anastassiou, G.A., Intelligent Systems: Approximation by Artificial Neural Networks. Springer-Verlag Berlin Heidelberg, (2011). 
\bibitem{banahc}  Anastassiou, G.A.,  Banach Space Valued Neural
Network, Springer Cham, Swirtzerland, (2023).
\bibitem{ARMSTRONG} Armsrtrong , J.S., Fildes, R.,  Making progress in forecasting, International Journal of Forecasting,  22(3), 433-441, (2006)
\bibitem{AVCI} Avcı, E., Forecasting daily and sessional returns of the ISE-100 index with neural network models, Doğuş Üniversitesi Dergisi, 8(2), 128-142,  (2007).
\bibitem{amatod} Amatod, P., Fariand, M., Montagnab, G., Morellia, M.J., Nicrosina , O., Treccanib, M., Pricing financial derivatives with neural networks, Physica A: Statistical Mechanics and its Applications, 338(1-2), 160-165,  (2004). 
\bibitem{BISHOP} Bishop, C., Neural Networks For Pattern Recognition, Clarendon Press, Oxford,  (1995).
\bibitem{BROOMHEAD} Bromhead, D.S., Lowe, D.,  Multivariable functional interpolation and adaptive networks, Complex Systems, 2, 321-355, (1988).
\bibitem{Cardaliaguet} Cardaliaguet, P., Euvrard, G.,  Approximation of a function and its derivative with a neural network. Neural Networks 5(2), 207-220, (1992).
\bibitem{CASDAGLI} Casdagli, M., Nonlinear prediction of chaotic time series, Physica, 35(D), 335-356, ( 1989).
\bibitem{Cybenko} Cybenko, G.,  Approximation by superpositions of sigmoidal functions. Math. Control Signals Syst. 2 (4), 303-314, (1989).
\bibitem{CHUNG} Chung, K.C.,  Tan, S.S., Holdsworth, D.K., Insolvency prediction model using multivariate discriminant analysis and artificial neural network for the finance industry in New Zealand, International Journal of Business and Management, 39(1), 19-28, ( 2008).
\bibitem{cIFTER} Çifter, A., Özün, A.,  Modelling long-term memory effect in stock prices, Studies in Economics and Finance, 25(1), 38-48 (2008).
\bibitem{FAUSETT} Fausett, L.,  Fundamentals of Neural Networks: Architectures, Algorithms, and Applications, New Jersey: Prentice Hall, (1994).
\bibitem{Seppo}  Linnainmaa, S., The representation of the cumulative rounding error of an algorithm as a Taylor expansion of the local rounding errors, (In Finnish), Master's Thesis, Department of Computer Science, University of Helsinki, Helsinki, Finland, (1970).
\bibitem{MISNKY} Misnky, M., Pappert, S., Perceptrons, MIT Press, Cambridge, MA., (1969).
\bibitem{FAALSIDE}  Niranjan, M.,  Faalside, F., Neural networks and radial basis functions in classifying static speech patterns, Computer Speech and Language, 4(3), 275-289, (1990).
\bibitem{CURTIS}  Özün, A.,  Hanias, M.P.,  Curtis, G., A  chaos analysis for Greek and Turkish equity markets, Euromed Journal of Business, 5(1), 101-118, (2010).
\bibitem{PARKER} Parker, D.B., Learning-logic, Technical Report No. 47, Center for Computational Research in Economics and Management Science., MIT., (1985).
\bibitem{POWELL} Powell, M.J.D., Radial basis functions for multivariable interpolation: A review.  Algorithms for Approximation, Oxford, Clarendon Press, 143-167, ( 1987).
\bibitem{HINTON} Rumelhart, D.E., Hinton, G.E., Williams, R.J.,  Learning internal representations by error propagation.  Parallel Distributed Processing, Vol 1., MIT Press, Cambridge, MA., (1986).
\bibitem{rozenbaltt}  Rosenblatt, F., The Perceptron:A Perceiving and Recognizing Automaton, Report 85-460-1, Cornell Aeronautical Laboratory, Buffalo, New York, (1957).
\bibitem{cantur} Turkun, C.,  Duman, O., Modified neural network operators and their convergence properties with summability methods. RACSAM, 114-132, (2020).
\end{thebibliography}
\end{document}